\documentclass[11pt]{amsart}
\usepackage{amsfonts, bm}
\usepackage{amssymb} 
\usepackage{mathrsfs}
\usepackage{hyperref}
\usepackage{amsmath}
\usepackage{tcolorbox}

\newtheorem{theorem}{Theorem}[section]
\newtheorem{lemma}[theorem]{Lemma}

\newtheorem{corollary}[theorem]{Corollary}

\newcommand\bes{\begin{eqnarray}}
	\newcommand\ees{\end{eqnarray}}
\newcommand\R{\mathbb R}
\numberwithin{equation}{section}
\theoremstyle{plain}
\newtheorem*{theorem*}{Theorem A}

\newcommand\bess{\begin{eqnarray*}}
	\newcommand\eess{\end{eqnarray*}}
\newcommand{\lf}{\left}
\newcommand{\rr}{\right}

\begin{document}
	
	\title[Convergence to a receding wave]{Convergence to a receding wave in a monostable\\  free boundary problem}
	\author[H. Cao, Y. Du, W. Ni ]{Hongkai Cao$^\dag$,  Yihong Du$^\ddag$ and Wenjie Ni$^{\ddag}$}
	\thanks{\hspace{-.5cm}
				\mbox{\ \ $^{\dag}$} School of Mathematics, Shandong University, Jinan, China.
		\\
		\mbox{\ \ $^\ddag$} School of Science and Technology, University of New England, Armidale, NSW 2351, Australia.
		\\
		\mbox{\small \ \ \ \  Emails:} chk@mail.sdu.edu.cn (H. Cao),\ ydu@une.edu.au (Y. Du),\ wni2@une.edu.au (W. Ni)}
	\thanks{ The research of Y. Du and W. Ni was supported by the Australian Research Council.}
	
	\date{\today}

	\begin{abstract}
		We study a monostable reaction-diffusion equation of the form $u_t=du_{xx}+f(u)$ over a semi-infinite spatial domain $[g(t),\infty)$, with  $x=g(t)$ the free boundary whose evolution is governed by
		equations  derived from a ``preferred population density'' principle, which postulates that the species with population density $u(t,x)$ and population range $[g(t),\infty)$ maintains a certain density $\delta$ at the habitat edge $x=g(t)$.
		In the ``high-density'' regime, where $\delta$ exceeds the carrying capacity of the favourable environment represented by a monostable function $f(u)$, it is known (see \cite{DLNS} for the case of  a bounded population range $[g(t), h(t)]$) that for large time, the front retreats as time advances. In this work, the unboundedness of the population range $[g(t),\infty)$ allows us to prove that, as time $t$ converges to infinity,  the free boundary $x=g(t)$ converges to $\infty$ with a constant asymptotic speed $c(\delta)>0$ determined by an associated  semi-wave problem, and the population density $u(t,x)$ has the property that $u(t,x+g(t))$ converges uniformly  to $q_{c(\delta)}(x)$, the semi-wave profile function associated with the speed $c(\delta)$. It turns out that in the retreating situation considered here, some key techniques developed for advancing fronts in related free boundary models do not work anymore. This difficulty is overcome here by a  ``touching method", which uses  a family of  lower and upper solutions constructed from semi-waves of some carefully designed auxiliary problems to  touch the  solution $u(t,x)$ at the moving boundary $x=g(t)$, thereby generating a setting where the comparison principle can be used to obtain the desired estimates for $g'(t)$ and $u(t,x)$.  We believe this method will find applications elsewhere.		\bigskip
		
		\medskip
		
		
	\end{abstract}
	
	\maketitle
	%\tableofcontents

	\section{Introduction and basic results}
	\setcounter{equation}{0}

	Reaction-diffusion equations of the form
	\begin{equation}\label{RD}
	u_t=du_{xx}+f(u), \ t>0,\ x\in\R,
	\end{equation}
	 particularly those of the Fisher-KPP type (\cite{Fisher1937, KPP1937}), have been successfully used to model  various biological and ecological phenomena, such as the spatial spread of invasive species. In such context, $u(t,x)$ stands for the population density of  a new or invasive species at time $t$ and spatial location $x$. If the initial population is localised in space, namely $u(0,x)$ is a nonnegative continuous function with non-empty but bounded support, then for a wide class of growth functions $f(u)$, including those of monostable type and bistable type, the dynamics of this model is largely determined by the associated traveling wave solutions (\cite{Fisher1937, KPP1937, AW1978, FM}), which are solutions to \eqref{RD} of the form $u(t,x)=\phi(x-ct)$, and therefore satisfies
	 \begin{equation}\label{tw}
	 d\phi''+c\phi'+f(\phi)=0 \mbox{ for } x\in\R.
	 \end{equation}
	 If $f(u)$ is of monostable type, namely
	\[
	(\mathbf{f_m}): \begin{cases}
		f \text{ is } C^{1}, \  f > 0 \text{ in } (0,1), \quad f < 0 \text{ in } (1, \infty), \\
		f(0) = f(1) = 0, \ f'(0) > 0 > f'(1),
	\end{cases}
	\]
then it is well known (see \cite{AW1978}, and also \cite{Fisher1937, KPP1937} for a narrower class of $f$) that
	there exists $c^*>0$ such that for every $c\geq c^*$, \eqref{tw} has a unique solution $\phi_c(x)$ satisfying 
	\[
	\phi_c(-\infty)=1,\ \phi(0)=1/2,\ \phi_c(+\infty)=0,
	\]
	and no such solution exists when $c<c_0$. Moreover, the constant $c_0$ serves as the propagation speed of the population modelled by \eqref{RD} with a localized initial population $u(0,x)$. More precisely, it follows from \cite{AW1978} that, if $u(t,x)$ is the unique solution of \eqref{RD} with such an initial function $u(0,x)$, then, for any small $\epsilon>0$,
	\[
	\lim_{t\to\infty} u(t,x)=\begin{cases} 1 \mbox{ uniformly for } |x|\leq (c_0-\epsilon)t,\\
	0 \mbox{ uniformly for } |x|\geq (c_0+\epsilon) t.
	\end{cases}
	\]
	This implies that if we fix a small $\tilde\epsilon>0$ and use 
	\[
	\Omega_{\tilde\epsilon}(t):=\{x: u(t,x)\geq \tilde\epsilon\}
	\]
	 as an approximation of the population range at time $t$,
	then $\Omega_{\tilde\epsilon}(t)$ propagates to the entire $\R$ with asymptotic speed $c_0$. Furthermore, there exists $\epsilon_+(t)$ and $\epsilon_-(t)$ satisfying $\epsilon_{\pm}(t)=o(t)$ as $t\to\infty$ such that
	\[
	\begin{cases}
	u(t,x)-\phi_{c_0}(x-c_0t +\epsilon_+(t))\to  0 \mbox{ uniformly for $x\geq 0$,}\\
	u(t,x)-\phi_{c_0}(c_0t+\epsilon_-(t)-x)\to 0 \mbox{ uniformly for } x\leq 0,
	\end{cases} \mbox{ as $t\to\infty$.}
	\]
	If $f(u)$ satisfies additionally the KPP condition $f(u)\leq f'(0)u$ for $u\geq 0$ (\cite{KPP1937}), then $c_0=2\sqrt{ f'(0) d}$.
	
	These classical results have been extended and further developed along various lins in the last several decades. Let us explain in particular one such development initiated in \cite{DuLin2010}, where in order to provide a precise population range in the model, \eqref{RD} was modified into a free boundary problem of the form
\begin{equation}\label{dulin}
		\left\{
		\begin{array}{ll}
			u_{t} - d u_{xx} = f(u), & t > 0, \quad  g(t)<x<h(t), \\
			u(t, g(t)) = u(t, h(t))=0, & t > 0, \\
			g'(t) = -\mu u_{x}(t, g(t)), & t > 0, \\
			h'(t) = -\mu u_{x}(t, h(t)), & t > 0, \\
			-g(0) = h(0)=h_0, \quad u(0, x) = u_{0}(x), &  -h_0\leq x\leq h_{0},
		\end{array}
		\right.
	\end{equation}	
	where $h_0$ and $\mu$ are given positive constants. This model provides a precise population range $[g(t), h(t)]$ for any given 
	 time $t>0$, without involving an artificial constant $\tilde\epsilon$ as in $\Omega_{\tilde\epsilon}(t)$ for \eqref{RD}.  The model \eqref{dulin} exhibits a spreading-vanishing dichotomy for its long-time dynamics (\cite{DuLin2010, DuLou}): As $t\to\infty$, either $[g(t), h(t)]$ converges to a finite interval $[g_\infty, h_\infty]$ and $u(t,x)\to 0$ uniformly (the vanishing case), or $[g(t), h(t)] \to \mathbb R$ and $u(t,x)\to 1$ (the spreading case). Moreover, in the latter case,  the long-time dynamics of \eqref{dulin} is also determined by the associated traveling wave solutions, which are called semi-waves, consisting of a unique pair $(c, q(z))$ with $c>0$ such that 
	\begin{equation*}%\label{sw}
		\left\{
		\begin{array}{l}
			dq'' - c q' + f(q) = 0 \quad \text{for} \quad z \in (0, \infty), \\
			q(0) = 0, \quad q(\infty) = 1, \quad q'(0)=c/\mu;
		\end{array}
		\right.
	\end{equation*}
more precisely, when spreading happens, the propagation of \eqref{dulin} is characterized by (\cite{DMZ})
	\begin{align*}
 	&  \lim_{t \to \infty} h'(t) = c, \quad \lim_{t \to \infty} g'(t) = -c,\\
 		&\lim_{t \to \infty} \sup_{x \in [0, h(t)]} |u(t, x) - q( h(t) - x)| = 0,\\
 			&\lim_{t \to \infty} \sup_{x \in [g(t), 0]} |u(t, x) - q( x - g(t))| = 0.
 \end{align*}
	
	A shortcoming of \eqref{dulin} as a model for species spreading  is that the population range $[g(t), h(t)]$ always expands as time increases.
	In the real world, the population range may expand as well as shrink when time increases. For example, it has been observed that in the cane toads invasion in Australia, 
	 temporary front retreats occurred in some cold climate areas  of New South Wales \cite{M-BI}. 
	 
	 This shortcoming of \eqref{dulin} has been eliminated by a further modification of the model given in \cite{Du2024}, which considers the problem
	 \begin{equation}\label{a}
\begin{cases}
u_t-du_{xx}=f(u),  & t>0,\; g(t)<x<h(t),\\
u(t,g(t))=u(t,h(t))=\delta, & t>0,\\
g'(t)=-\frac{d}{\delta}u_{x}(t,g(t)), &  t>0,\\
h'(t)=-\frac{d}{\delta}u_x(t,h(t)), &  t>0,\\
-g(0)=h(0)=h_0, u(0,x)=u_0(x), & -h_0\leq x\leq h_0.
\end{cases}
\end{equation}
Here
 $f$ is assumed to satisfy ${\bf (f_m)}$, and
 the constant $\delta\in(0,1)$ in \eqref{a} represents a preferred density of the species, and the equations governing the evolution of the free boundaries $x=g(t)$ and $x=h(t)$, namely the second, third and fourth equations in \eqref{a}, can be deduced from the biological assumption that the species maintains its preferred  density $\delta$ at the range boundary by advancing or retreating the fronts; see \cite{Du2024} for a detailed deduction and more background. This assumption means that for members of the species near the range boundary, their movement in space is governed by two factors: (a) random movement similar to all other members, and (b) advance or retreat to keep the density at the front at the preferred level $\delta$.
 
It is easy to see that in \eqref{a}, the fronts $x=h(t)$ and $x=g(t)$ 
 may advance or retreat as time increases, depending on how the population $u(t,x)$ is distributed in $x\in [g(t), h(t)]$ at time $t$.
It was shown in \cite{Du2024} that the unique solution $(u(t,x), g(t), h(t))$ of  \eqref{a} exhibits successful spreading for all admissible initial data\footnote{namely for every $h_0>0$ and $u_0\in X(h_0):=\{\phi\in C^2([-h_0, h_0]): \phi(\pm h_0)=\delta, \ \phi>0 \ \text{in} \ [-h_0, h_0]\}$.}, and
 \begin{align*}
 	&\lim_{t \to \infty} h'(t) = c_\delta^*, \quad \lim_{t \to \infty} g'(t) = -c_\delta^*,\\
 		&\lim_{t \to \infty} \sup_{x \in [0, h(t)]} |u(t, x) - q_\delta^*( h(t) - x)| = 0,\\
 			&\lim_{t \to \infty} \sup_{x \in [g(t), 0]} |u(t, x) - q_\delta^*( x - g(t))| = 0,
 \end{align*}
where $(c_\delta^*, q_\delta^*)$ is the unique solution pair $(c,q)$ of 
		\begin{equation}\label{semi1}
			\begin{cases}
				dq'' - cq' + f(q) = 0, \quad q > 0 \quad \text{in } (0,\infty), \\
				q(0) = \delta, \quad q(\infty) = 1, \quad q'(0) = \frac{c \delta}{d},\quad q'>0 \mbox{ in } [0,\infty).
			\end{cases}
		\end{equation}
	
If $\delta>1$ in \eqref{a}, then it follows from the recent work \cite{DLNS} that the population modelled by \eqref{a} vanishes, namely there exists some $\xi_*\in\R$ such that
\[
\lim_{t\to\infty} g(t)=\lim_{t\to\infty} h(t)=\xi_*,\ \lim_{t\to\infty} u(t,x)=\delta \mbox{ uniformly in } x\in[g(t), h(t)].
\]
Thus in such a case, the population range shrinks to one point and the population vanishes by losing all its habitat. 
	
	In a related but different direction, the authors of \cite{Simpson2022} recently investigate a Fisher-KPP model of the form
		 \begin{equation}\label{b}
\begin{cases}
u_t-u_{xx}=u(1-u),  & t>0,\; 0<x<s(t),\\
u_x(t,0)=0,\ u(t,s(t))=\delta, & t>0,\\
h'(t)=-\kappa u_x(t,s(t)), &  t>0,\\
s(0)=h_0, u(0,x)=u_0(x), & 0\leq x\leq h_0,
\end{cases}
\end{equation}
	where $\delta\in (0,1)$ but $\kappa$ may take positive as well as negative values, and the initial function $u_0(x)$ is piece-wisely linear: $u_0(x)$ equals $1$ over $[0,\beta]$ for some $\beta\in (0, h_0)$, and $u_0(h_0)=\delta$. The numerical simulation results in \cite{Simpson2022} indicate that when $\kappa>0$, then $u(t,x)$ behaves like a propagating semi-wave as $t$ increases, much as what was proved for \eqref{a} in \cite{Du2024}, but when $\kappa<0$ and $h_0>0$ is large, the density function $u(t,x)$ behaves like a retreating wave before $s(t)$ reaches $0$. The retreating wave to \eqref{b} with $\kappa<0$ was carefully examined in \cite{Simpson2022}.
	
	In this paper, we consider a retreating wave problem related to but different from both  \cite{DLNS} and \cite{Simpson2022}. Our problem here arises from an on-going work on a model for propagation in a shifting environment--the mathematical techniques developed here will be one of the main technical ingredients  in this forthcoming work.
	
	To be more precise,  in this paper we consider the following problem
	\begin{equation}\label{1.1}
		\left\{
		\begin{array}{ll}
			u_{t} - d u_{xx} = f(u), & t > 0, \quad x > g(t), \\
			u(t, g(t)) = \delta, & t > 0, \\
			g'(t) = -\dfrac{d}{\delta} u_{x}(t, g(t)), & t > 0, \\
			g(0) = g_{0}, \quad u(0, x) = u_{0}(x), &  x\geq g_{0},
		\end{array}
		\right.
	\end{equation}
	where $d>0$, $\delta>1$, $g_0\in \mathbb{R}$ and $f:[0, \infty) \rightarrow \mathbb{R} $ satisfies ${\bf (f_m)}$. The initial function \( u_{0}(x) \) is assumed to belong to
	\[
	\mathcal{X}(g_{0}) := \left\{ \phi \in C^{2}\bigl( [g_{0},\infty) \bigr) : \|\phi\|_{C^2([g_0, \infty))}<\infty, \inf_{x\geq g_0}\phi(x) > 0,   \,  \phi(g_{0}) = \delta \right\}.
	\]
	
	\medskip
	
		Our first result concerns the global existence and uniqueness of the solution, which will be proved by further developing the existing techniques in \cite{Du2024, DLNS}.
	\begin{theorem}\label{th1.1}
		Suppose that $(\mathbf{f_m})$ holds, $u_{0} \in \mathcal{X}(g_{0})$,  $ \alpha \in (0, 1)$ and $\delta>1$. Then, \eqref{1.1} admits a unique solution
		\begin{align*}
			(u,  g) \in C^{1+\frac{\alpha}{2}, 2+\alpha}(\Omega_{\infty}) \times C^{1+\frac{\alpha}{2}}((0, \infty)),
		\end{align*}
		where $\Omega_{\infty}:=\left\{(t, x) \in \mathbb{R}^{2}: t \in(0, \infty), \,x \in[g(t), \infty)\right\}.$
	\end{theorem}
	
	Our main result below concerns the long-term dynamics of \eqref{1.1}, which shows that it is determined by the associated retreating wave. 	
	\begin{theorem}\label{th1.2} Under the conditions of Theorem \ref{th1.1}, let
		 $(u,g)$ be the unique solution of \eqref{1.1} with $u_{0} \in \mathcal{X}(g_{0})$. Then 
		\begin{align}
			&\lim_{t \to \infty}g(t)=\infty, \quad \lim_{t\rightarrow \infty} g'(t)=c(\delta),\label{1.2}\\
			&	\lim_{t \to \infty}\sup_{x \geq g(t)}|u(t,x)-q^*(x-g(t))|=0,\label{1.3}
		\end{align}
		where $c(\delta):=-c^*>0$ and $(c^*, q^*(z))$ is the unique pair $(c, q(z))$ satisfying
			\begin{equation}\label{eq:part2}
				\begin{cases}
					d q'' - c q' + f(q) = 0,\ q'<0 \quad \mbox{for } z > 0, \\
					q(0) = \delta, \quad q'(0) = \dfrac{c \delta}{d}, \quad  q(\infty) = 1.	
				\end{cases}
			\end{equation}
		 Moreover, \( c(\delta) \) is increasing in \( \delta \) for \( \delta > 1 \), and \[ \lim_{\delta \to 1^{+}} c(\delta) = 0 .\]
	\end{theorem}
	This result indicates that for large time, the population range $[g(t), \infty)$ shinks as time increases, with $x=g(t)$ converging to $\infty$ with a constant asymptotic speed $c(\delta)$ as $t\to\infty$. At the same time,
	the shifted density function $u(t,x+g(t))$ converges to the retreating semi wave profile $q^*(x)$.

	The rest of this paper is organized as follows. Section~2 is devoted to the proof of the well-possedness result Theorem \ref{th1.1}, by further developing existing techniques in \cite{Du2024, DLNS}. The main novelty of the paper is contained in Section 3, where we complete the proof of Theorem \ref{th1.2} in three subsections. In particular, subsection 3.2 presents several results on semi-waves, which are used in subsection 3.3, where  the proof of Theorem~\ref{th1.2} is completed by employing a new technique, called the ``touching method", which uses families of upper and lower solutions obtained from certain carefully designed auxiliary problems in subsection 3.2. 	In order not to interrupt the  flow of  thoughts, the proofs of the results on semi-waves in subsection 3.2 are postponed to Section 4.
	
	\section{Proof of Theorem \ref{th1.1}}
	
	\subsection{Local existence and uniqueness}
	
	In the following local existence result, we consider a more general function $f(u)$.
	\begin{theorem}\label{th2.1}
		Suppose that $f$ is $C^{1}$ and satisfies $f(0)=0$. Then, for any given $u_{0} \in \mathcal{X}(g_{0})$ and $\alpha \in (0, 1)$, there
		is a constant $T > 0$ such that problem \eqref{1.1} admits a unique solution
		\[(u, g) \in C^{(1+\alpha) / 2,1+\alpha}\left(\overline{\Omega}_{T}\right) \times C^{1+\alpha / 2}([0, T]),\]	where $\Omega_{T}=\left\{(t, x) \in \mathbb{R}^{2}: t \in(0, T], x \in[g(t),\infty)\right\}$.
	\end{theorem}
	
	\begin{proof}
		We will complete the proof in three steps. In Step 1, we rewrite the problem into one with fixed straight boundaries and reduce the existence problem into a fixed point problem.  In Step 2, we establish the local existence and uniqueness of the solution by applying the contraction mapping theorem combined with an extension technique. Finally, in Step 3, we employ Schauder theory to obtain additional smoothness of the solution.

		\textbf{Step 1.} We straighten the lateral boundary of the domain $\Omega_{T}$ and reduce the existence problem to a fixed point problem, where we make use of an extension technique combined with $L^{p}$ theory and Sobolev embedding theorems.  
		
		Define $U(t, y) := u(t, y+g(t))$ for $(t, y+g(t))=(t, x) \in \Omega_{T}$. Then \eqref{1.1} for $t \in (0, T]$ is equivalent to  
		\begin{equation}\label{4.1}
			\left\{
			\begin{array}{ll}
				U_{t} - g'(t) U_y - d U_{yy} = f(U), & t > 0,  y > 0, \\
				U(t,0) = \delta, & t > 0, \\
				U_{y}(t, 0) = -\dfrac{\delta}{d} g'(t), & t > 0, \\
				U(0, y) = u_{0}(y + g_0), & y \geq 0.
			\end{array}
			\right.
		\end{equation}

		Let $H := \max \{1, |g^{0}| \}$ where $g^{0} = -\frac{d}{\delta}u_{0}^{\prime}(g_{0})$. For $T \in (0, 1]$ and $D_{T} := \{(t, y): 0 \leq t \leq T, y \geq 0\}$, define the function spaces  
		\[
		\begin{array}{l}
			X_{U, T} := \left\{U \in C\left(D_{T}\right): U(0, y) = u_{0}(y),\, \| U - u_{0} \|_{C\left(D_{T}\right)} \leq 1\right\}, \\[2mm]
			X_{g, T} := \left\{g \in C^{0,1}([0, T]): g(0) = g_{0},\, g^{\prime}(0) = g^{0},\, \| g^{\prime} - g^{0} \|_{L^{\infty}([0, T])} \leq H\right\}.
		\end{array}
		\]
		The product space $X_{T} := X_{U, T}\times X_{g,T}$ is a complete metric space with the metric:  
		\[
		d\left((U_{1}, g_{1}),\, (U_{2}, g_{2})\right) := \| U_{1} - U_{2} \|_{C(D_{T})} + \| g_{1}^{\prime} - g_{2}^{\prime} \|_{L^{\infty}([0, T])}.
		\]  
		For $0 < T < 1$, define the subspace $X_{1}^{T} = X_{U, 1}^{T}\times X_{g,1}^T$ where  
		\[
		\begin{array}{l}
			X_{U, 1}^{T} := \left\{U \in X_{U, 1}: U(t, y) = U(T, y) \text{ for } T \leq t \leq 1\right\}, \\[2mm]
			X_{g, 1}^{T} := \left\{g \in X_{g, 1}: g(t) = g(T) \text{ for } T \leq t \leq 1\right\}.
		\end{array}
		\]
		Since functions  in $X_{T}$ can be extended to $X_{1}^{T}$ via the above definitions, we will henceforth identify $X_{T}$ with $X_{1}^{T}$.  
		
		For each $(U, g) \in X_{T} = X_{1}^{T}$, consider  
		\begin{equation}\label{2.2}
			\left\{\begin{array}{ll}
				\bar{U}_{t} - d \bar{U}_{yy} - g' \bar{U}_{y} = f(U), & 0 < t \leq 1, \, y > 0, \\
				\bar{U}(t,0) = \delta, & 0 < t \leq 1, \\
				\bar{U}(0, y) = u_{0}(y + g_0), & y \geq 0.
			\end{array}\right.
		\end{equation}
		Since $\|g'(t)\|_{L^{\infty}([0, T])} \leq H + |g^0|$ and $u_{0} \in C^{2}\left([g_{0}, \infty)\right)$ with $u_{0}(g_{0}) = \delta$, we apply $L^{p}$ theory to \eqref{2.2} and the Sobolev embedding theorem \cite{GML} to conclude that for $\alpha \in (0,1)$, \eqref{2.2} admits a unique solution $\bar{U}$ satisfying  
		\begin{equation}\label{2.6}
			\|\bar{U}\|_{C^{\frac{1+\alpha}{2}, 1+\alpha}\left([0, 1] \times [m, m+1]\right)} \leq C_{1} \|\bar{U}\|_{W_{p}^{1,2}\left([0, 1] \times [m, m+1]\right)} \leq K_{1},
		\end{equation}
		where $p > 3/(2-\alpha)$. The constant $K_{1}$ depends on $p$, $\| f(U) \|_{L^{\infty}(D_{1})}$, $\| u_{0} \|_{C^{2}([g_{0}, \infty))}$, $g_{0}$, $g^{0}$, and $C_{1}$, but not on $m$; $C_{1}$ depends on $D_{1}$ and $\alpha$.  
		
		Define $\bar{g}(t)$ via $\bar{g}'(t) = -\frac{d}{\delta} \bar{U}_{y}\left(t, 0\right)$. Then $\bar{g}' \in C^{\frac{\alpha}{2}}\left([0, 1]\right)$ with  
		\begin{equation}\label{2.7}
			\|\bar{g}'\|_{C^{\frac{\alpha}{2}}\left([0, 1]\right)} \leq K_{2},
		\end{equation}
		where $K_{2}$ depends on $K_{1}$.  
		Define the mapping $\mathcal{F}: X_{T} \to C\left(D_{1}\right) \times C([0, 1])$ by  
		\[
		\mathcal{F}(U, g) = (\bar{U}, \bar{g}).
		\]
		Set $\tilde{\mathcal{F}}(U, g) := \left.\mathcal{F}(U, g)\right|_{X_{T}}$. Then $(U, g)$ is a fixed point of $\tilde{\mathcal{F}}$ if and only if it solves \eqref{2.2}, which is equivalent to \eqref{1.1} for $t \in [0, T]$.
		
		\textbf{Step 2.} We show that $\widetilde{\mathcal{F}}$ is a contraction mapping for sufficiently small $T > 0$. 
		
		For $T <\min\left\{1,\, K_{1}^{-\frac{2}{1+\alpha}},\, K_{2}^{-\frac{2}{\alpha}}\right\}$, we have
		\[
		\begin{aligned}
			&\|\overline{U} - u_{0}\|_{C(D_T)} \leq T^{\frac{1+\alpha}{2}} \|\overline{U}\|_{C^{0,\frac{1+\alpha}{2}}(D_T)} \leq T^{\frac{1+\alpha}{2}} \|\overline{U}\|_{C^{0,\frac{1+\alpha}{2}}(D_1)} \leq K_1 T^{\frac{1+\alpha}{2}} \leq 1, \\
			&\|\overline{g}'(t) - g^{0}\|_{L^\infty([0,T])} \leq T^{\frac{\alpha}{2}} \|\overline{g}'\|_{C^{\frac{\alpha}{2}}([0,T])} \leq T^{\frac{\alpha}{2}} \|\overline{g}'\|_{C^{\frac{\alpha}{2}}([0,1])} \leq K_2 T^{\frac{\alpha}{2}} \leq H,
		\end{aligned}
		\]
		which implies that $\widetilde{\mathcal{F}}$ maps $X_T$ into itself.
		
		Next, we prove that $\widetilde{\mathcal{F}}$ is a contraction on $X_T$ for sufficiently small $T > 0$. Let $(U_i, g_i) \in X_T$ for $i=1,2$, and set $W := \overline{U}_1 - \overline{U}_2$. Then $W$ satisfies
		\begin{equation}\label{2.8}
			\left\{
			\begin{array}{ll}
				W_t - d W_{yy} - g_1'(t) W_y = \Psi, & y > 0,\ 0 < t \leq 1, \\
				W(t, 0) = 0, & 0 < t \leq 1, \\
				W(0, y) = 0, & y \geq 0,
			\end{array}
			\right.
		\end{equation}
		where
		\[
		\Psi := \left(g_1'(t) - g_2'(t)\right) \overline{U}_{2,y} + f(U_1) - f(U_2).
		\]
		For any $p > 1$ and integer $m \geq 1$,
		\begin{equation}\label{2.11}
			\begin{aligned}
				\|\Psi\|_{L^p([0,1] \times [m, m+1])} 
				&\leq \|g_1' - g_2'\|_{L^\infty([0,1])} \|\overline{U}_{2,y}\|_{L^p([0,1] \times [m, m+1])} \\
				&\quad + \|f(U_1) - f(U_2)\|_{L^p([0,1] \times [m, m+1])} \\
				&\leq S \left( \|U_1 - U_2\|_{C([0,1] \times [m, m+1])} + \|g_1' - g_2'\|_{L^\infty([0,1])} \right),
			\end{aligned}
		\end{equation}
		where the constant $S$ depends on the domain $D_1$, $K_1$, and the Lipschitz constant of $f$.
		
		Applying $L^p$ estimates to \eqref{2.8} and using Sobolev embedding, we obtain
		\begin{equation}\label{2.12}
			\begin{aligned}
				\|W\|_{C^{\frac{1+\alpha}{2},1+\alpha}([0,1] \times [m, m+1])} 
				&\leq C_1 \|W\|_{W_p^{1,2}([0,1] \times [m, m+1])} \\
				&\leq K_3 \left( \|U_1 - U_2\|_{C([0,1] \times [m, m+1])} + \|g_1' - g_2'\|_{L^\infty([0,1])} \right),
			\end{aligned}
		\end{equation}
		where $K_3$ depends on $p$, $D_1$, $S$, and $C_1$.
		From the definition of $\overline{g}_i(t)$, we have
		\begin{equation}\label{2.13}
			\left\|\overline{g}_1' - \overline{g}_2'\right\|_{C^{\frac{\alpha}{2}}([0,1])} \leq \frac{d}{\delta } \left\| \overline{U}_{1,y} - \overline{U}_{2,y} \right\|_{C^{0,\frac{\alpha}{2}}([0,1] \times [0,1])}.
		\end{equation}
		Combining \eqref{2.12} and \eqref{2.13} yields
		\begin{align*}
			&\left\|\overline{U}_1 - \overline{U}_2\right\|_{C^{\frac{1+\alpha}{2},1+\alpha}(D_1)} + \left\|\overline{g}_1' - \overline{g}_2'\right\|_{C^{\frac{\alpha}{2}}([0,1])} \\
			&\quad \leq K_4 \left( \|U_1 - U_2\|_{C(D_1)} + \|g_1' - g_2'\|_{L^\infty([0,1])} \right),
		\end{align*}
		where $K_4$ depends on $d$, $\delta$, $g_0$, and $K_3$.
		If we take  $T=\min \left\{\frac{1}{2},\, K_{1}^{\frac{-2}{1+\alpha}},\, K_{2}^{\frac{-2}{\alpha}}, \,K_{4}^{\frac{-2}{\alpha}}\right\}$, then we have
	\[ \begin{aligned}
		& \left\Arrowvert \bar{U}_{1} - \bar{U}_{2} \right\Arrowvert_{C(D_T)} + \left\Arrowvert \bar{g}_{1}^{\prime} - \bar{g}_{2}^{\prime} \right\Arrowvert_{C([0, T])} \\
		\leq & \,T^{\frac{1+\alpha}{2}} \left\Arrowvert \bar{U}_{1} - \bar{U}_{2} \right\Arrowvert_{C^{\frac{1+\alpha}{2}, 1+\alpha}(D_1)}
		+ T^{\frac{\alpha}{2}} \left\Arrowvert \bar{g}_{1}^{\prime} - \bar{g}_{2}^{\prime} \right\Arrowvert_{C^{\frac{\alpha}{2}}([0, 1])} \\
		\leq & \,\frac{1}{2} \left( \left\Arrowvert U_{1} - U_{2} \right\Arrowvert_{C(D_1)} + \left\Arrowvert g_{1}^{\prime} - g_{2}^{\prime} \right\Arrowvert_{L^{\infty}([0, 1])} \right) \\
		= &\, \frac{1}{2} \left( \left\Arrowvert U_{1} - U_{2} \right\Arrowvert_{C(D_T)} + \left\Arrowvert g_{1}^{\prime} - g_{2}^{\prime} \right\Arrowvert_{L^{\infty}([0, T])} \right),
	\end{aligned}\]
		which shows that $ \tilde{\mathcal{F}} $ is a contraction mapping on $ X_{T}$. Therefore, it has a unique fixed point  $(U, g) $ in $ X_{T}$. The maximum principle implies that $ U>0$  in $ [0, T]$ $\times\left[0, \infty\right)$.

		\textbf{Step 3}. We apply the Schauder theory to obtain additional regularity for the solution of \eqref{4.1} in $ \left[0, \infty\right) \times(0, T] $.
		
		From Step 2 we know that  $U \in C^{\frac{1+\alpha}{2}, 1+\alpha}\left(D_{T}\right)$ and $ g \in C^{1+\frac{\alpha}{2}}([0, T]) $. 
		Since $ u_{0} \in C^{2}\left(\left[g_{0}, \infty\right)\right)$, we cannot directly apply Schauder theory to \eqref{4.1}. To address this, we employ a standard technique involving a cutoff function. For any given small constant  $\varepsilon>0$, we choose  $\xi^{*} \in C^{\infty}([0, T]) $ such that
			\[\xi^{*}(t)=\left\{\begin{array}{l}
			1, \text { for } t \in[2 \varepsilon, T], \\
			0, \text { for } t \in[0, \varepsilon] .
		\end{array}\right.\]
		We then define $ V:=U \xi^{*}$. Substituting into \eqref{4.1}, we obtain the following modified system
		\begin{equation}\label{2.15}
			\left\{\begin{array}{ll}
				V_{t}=d V_{y y}+g'V_{y}+F(t, y), & 0<t \leq T,\,y>0, \\
				V\left(t, 0\right)=\delta \xi^{*}(t),\,&0<t \leq T,\\
				V(0, y)=0,\, &  y \geq 0,
			\end{array}\right.
		\end{equation}
		where
		\[F(t, y)=U \xi_{t}^{*}-f(U) \xi^{*}.\]
		Since  $F \in C^{\frac{\alpha}{2}, \alpha}\left([0, T] \times\left[0,\infty\right)\right)$ and $	g^{\prime}(t) \in C^{\frac{\alpha}{2}, \alpha}\left(D_{T}\right)$, we can apply the Schauder estimate (see Theorem 5.14 in \cite{GML}) to \eqref{2.15} to deduce
		\[\begin{aligned}
			\Arrowvert U\Arrowvert _{C^{1+\frac{\alpha}{2} 2+\alpha}\left([2 \varepsilon, T] \times\left[ 0, \infty\right)\right)} & \leq\Arrowvert V\Arrowvert _{C^{1+\frac{\alpha}{2}, 2+\alpha}\left([0, T] \times\left[ 0, \infty\right)\right)} \\
			& \leq M\Arrowvert F\Arrowvert _{C^{\frac{\alpha}{2}, \alpha}\left([0, T] \times\left[ 0, \infty\right)\right)},
		\end{aligned}\]
		where $ M$  depends on $ \varepsilon$, $ g_{0}$, and $\alpha$.
		Since  $\varepsilon $ can be arbitrarily small, it follows that
		$g \in C^{1+\frac{1+\alpha}{2}}((0, T])$, $u \in C^{1+\frac{\alpha}{2}, 2+\alpha}\left(\Omega_{T}\right)$.
		Thus,  $(u, g)$  is a classical solution of \eqref{1.1} over $ \Omega_{T}$.	
	\end{proof}

	\subsection{Global existence}

	We follow the approach of \cite{Du2024, DLNS} to prove the global existence and uniqueness result.

	\begin{lemma}[\underline{A priori bounds}]\label{le2.7}
		Suppose that $\left(\mathbf{f_m}\right)$ holds, $(u,g)$ is a solution to \eqref{1.1} defined for $t\in [0,T]$ for some $T\in (0,\infty)$. Then there exist constants $C_{1}, C_{2} >0$ independent of $T$ such that
		\begin{align*}
			0 < u(t,x) &\leq C_1 \quad \text{for } t \in [0, T] \text{ and } x \in [g(t), \infty), \\
			|g'(t)|  &\leq C_2 \quad \text{for } t \in [0, T].
		\end{align*}
	\end{lemma}
	
	\begin{proof} We divide the proof into several steps.
	
		{\bf Step 1.} {We obtain the desired bound for $u$.}
		
		By the standard theory of parabolic equations, the following initial-boundary value problem in half line
		\[\left\{\begin{array}{ll}
			v_{t}-d v_{x x}=f(v), & t \in(0, T],\, x \in(g(t), \infty), \\
			v(t, g(t))=\delta, & t \in(0, T], \\
			v(0, x)=\inf_{x \geq g_0}u_0(x), & x \in\left[g_0, \infty\right),
		\end{array}\right.\]
		admits a unique solution  $v(t, x)$, which satisfies  $v(t, x)>0 $ for $ t \in(0, T]$ and $x \in[g(t), \infty) $. Applying the comparison principle, we deduce that $ u(t, x) \geq v(t, x)>0$  for  all $ t \in(0, T]$ and $x \in[g(t), \infty) $. This establishes the positivity of $u(t,x)$.

		Next, define  $C_{1}:= \|u_{0}\|_{L^\infty([g_0,\infty))}+1 $. Let  $w(t)$  be the unique solution of the ODE problem
		\[w^{\prime}=f(w),\, w\left(0\right)=C_{1} .\]
		From assumption $\left(\mathbf{f}\right)$,  it follows that $ w(t) \rightarrow 1 $ as $ t \rightarrow \infty $ and  $w(t)\leq m^{*} $ for all $ t\geq 0$. By the comparison principle, we have $u(t, x) \leq w(t) $ for  $t \in\left[0, T\right]$  and $  x \in[g(t), \infty)$. Thus, it is clear that
		\[0<u(t, x) \leq C_{1}\,\text { for }\, t \in [0, T] \text{ and } g(t)\leq x<\infty .\]
		{\bf Step 2.} {We construct a lower solution to obtain a lower bound for $g'(t)$.}
		
		For some constants $0<c<1$, $k>0$ and $m>0$ to be determined later, define 
		\begin{align*}\begin{cases} 
		\theta(s):=ce^{-s/c}-c \mbox{ for } s\geq 0, \\
			\underline u(t,x):=\delta\theta((x-g(t))m)+\delta \ \mbox{ for }\ t\geq 0,\,\, x\in [g(t),g(t)+k/m].
			\end{cases}
		\end{align*}
		 We will show that the following  hold for suitable choices of $c,\ k$ and $m$:
		\begin{equation}\label{1}
			\begin{cases}
				\underline u_t< d\underline u_{xx}+f(\underline u),\ \ &t\in (0,T],\, x\in [g(t),g(t)+k/m],\\
				\underline u(t,g(t)+k/m)< u(t,g(t)+k/m),\ \underline u(t, g(t))=\delta,\ & t\in [0,T],\\
				\underline u(0,x)\leq  u_0(x),\ &x\in [g_0, g_0+k/m],
			\end{cases}
		\end{equation}
		and 
		\begin{align}\label{2}
			g'(t)\leq -\frac{d}{\delta} \underline u_x(t, g)=dm,\ \ t\in [0,T].
		\end{align}
		
		Denote $c_0:=\min\{1/2, \inf_{x \geq g_0}u_0(x)\}$, 	choose $c\in (0,1)$  close to $1$, and 
		\begin{align*}
			k:=-c\ln \lf(\frac{c_0}{c\delta}+1-\frac{1}{c}\rr)>0.
		\end{align*} 
		Then, $\underline u(t,g(t)+k/m)=c_0$. 
		Consequently, we have
		\begin{align*}
			\underline u(t,g(t)+k/m)=c_0\leq u(t,x)\ \ {\rm for\ all}\ t\geq 0,\ x\in [g(t),\infty).
		\end{align*}
		By the definition of $\underline u$, we have $\underline u(t, g(t))=\delta$, and so the conditions in the second line of \eqref{1} are satisfied. 
		
		It is clear that  $\underline u_x(0, x)\leq -\delta m e^{-k/c}<-\max_{x\in [g_{0},g_{0}+1 ]}|u_0'(x)|$ for $x\in[g_{0}, g_{0}+k/m]$ provided  $m$ is chosen sufficiently large. 
				Since $\underline u(0,g_0)=u_0(g_0)=\delta$,  it follows that,
		for such $m$,  
		\begin{align*}
			\underline u(0,x)<u_0(x)\ \mbox{ for } \ x\in (g_0,g_0+k/m],
		\end{align*}
		which establishes the inequality in the third line of \eqref{1}.

		Next we prove the first inequality of \eqref{1} for $t\in [0, T_{1})$, where $T_{1}\leq T$ is  defined by
		\begin{align*}
			T_1=\sup\left\{t\in (0, T]: g'(s)<dm=-\frac{d}{\delta} \underline u_x(s, g(s))\ {\rm for\ all}\ s\in [0,t)\right\}.
		\end{align*}
		It is easily seen that for large $m>0$
		\begin{align*}
			g'(0)<dm= -\frac{d}{\delta} \underline u_x(0, g_{0}),
		\end{align*}
		and so $T_{1}$ is well defined and
		\begin{equation}\label{T1}
		g'(t)\leq dm \mbox{ for } t\in (0, T_1].
		\end{equation}
		 (In Step 4, we will  show  $T_{1}$ equals to $T$.)
		
		Clearly $0<\underline u(t,x)\leq\delta$, and by direct computation,
		\begin{align}\label{underline-u}
			\underline u_t=\delta m e^{-(x-g(t))m/c} g'(t),\ \ \underline u_x=-\delta m e^{-(x-g(t))m/c},\ \  \underline u_{xx}=\frac{\delta m^2}{c} e^{-(x-g(t))m/c}.
		\end{align}
		Using \eqref{T1} and \eqref{underline-u} we obtain
		\begin{align*}
			\underline u_t\leq d\delta m^2 e^{-(x-g(t))m/c} \mbox{ for } t\in (0, T_1], \ x\in [g(t), g(t)+k/m].
		\end{align*}
		Thus, if 
		\[ d\delta m^2 e^{-(x-g(t))m/c} <d \underline u_{xx}-K=d \frac{\delta m^2}{c} e^{-(x-g(t))m/c}-K,  \mbox{ with } K:=\max_{s\in [0,\delta]} |f(s)|,\]
		then the first inequality in  \eqref{1} holds with $T$ replaced by $T_1$. This inequality is satisfied provided
		\begin{align}\label{m}
			m>\sqrt{\frac{Kce^{k/c}}{d\delta(1-c)}}. 
		\end{align}
		Thus \eqref{1} and \eqref{2} hold for $t\in  (0, T]$ when the parameters are chosen as above, except that the first inequality in \eqref{1} and \eqref{2} are proved only for $t\in (0, T_1]$.
		We have to wait till Step 4 to show $T_1=T$, thus fully completing Step 2.
		
		{\bf Step 3.} We prove that \eqref{T1} implies
		\begin{align}\label{2.26}
				 g'(t) \geq -2M(2C_{1}-\delta)d/\delta \ \,\text{for} \,\, t\in [0,T_{1}] \mbox{ and some $M\gg 1$ independent of $T_1$.}
		\end{align}
		
	 Since  $f $ is $ C^{1}$  and $ f(0)=0 $, there exists a constant $L>0$, depending on  $C_{1}$,  such that $ f^{\prime}(u) \leq L$  for  $u \in\left[0,\,2 C_{1}\right] $. 
	 For some $ M>1 $ to be specified later, define
		\[\begin{aligned}
			&\Omega:=\left\{(t, x): 0\leq t\leq T_{1},\, g(t)<x<g(t)+M^{-1}\right\}, \\
			&\bar{u}(t, x):=\left(2 C_{1}-\delta\right)\left[2 M(x-g(t))-M^{2}(x-g(t))^{2}\right]+\delta.
		\end{aligned}\]
		Clearly,
		\[\begin{aligned}
			&\bar{u}(t, x)=\left(2 C_{1}-\delta\right)\left(1-[1-M(x-g(t))]^{2}\right)+\delta \leq 2 C_{1} \text { for }(t, x) \in \Omega ,\\
			&\bar{u}(t, g(t))=\delta, \,\bar{u}\left(t, g(t)+M^{-1}\right)=2 C_{1}>u\left(t, g(t)+M^{-1}\right) .
		\end{aligned}\]
		Moreover, for $ (t, x) \in \Omega$, using $g
		'(t)\leq md$ for $t\in[0,T_{1})$ we obtain
		\[
		\begin{aligned}
			\bar{u}_{t}-d \bar{u}_{x x}-f(\bar{u}) & =\left(2 C_{1}-\delta\right) g^{\prime}(t)\left[-2 M-2 M^{2}(g(t)-x)\right]+d\left(2 C_{1}-\delta\right) 2 M^{2}-f(\bar{u}) \\
			&\geq -2Mdm(2C_{1}-\delta)[1+M(g(t)-x)]+d(2C_{1}-\delta)2M^{2}-2LC_{1}\\
			&\geq -2Mdm(2C_{1}-\delta)+d(2C_{1}-\delta)2M^{2}-2LC_{1}\\
			& = 2 d M[M-m]\left(2 C_{1}-\delta\right)-2 L C_{1} \geq 0,
		\end{aligned}\]
		 provided that $M$ is chosen to satisfy
		\[M \geq\sqrt{\frac{L C_{1}}{d\left(2 C_{1}-\delta\right)}}+m .\]
		
		We next show that
		\[
		\bar{u}\left(0, x\right) \geq u_{0}(x)\text { for } x \in\left[g_{0},\, g_{0}+M^{-1}\right],
		\]
		provided $M$ is sufficiently large. Indeed, by choosing $M$ large enough such that  $M(2C_{1}-\delta)>\max_{x \in\left[g_{0}, g_0+1 \right]}|u'_{0}(x)|$, we have
		\[
		\bar{u}_{x}(0, x)\geq M\left(2 C_{1}-\delta\right)> u'_{0}(x) \,\,\text{for} \,\,x \in[g_{0},\,g_{0}+(2M)^{-1}].
		\]
		This and $\bar{u}\left(0, g_{0}\right) = u_{0}(g_{0})=\delta$ imply that for sufficiently large $M>0$, 
		\[
		\bar{u}\left(0, x\right) \geq u_{0}(x)\text { for } x \in\left[g_{0},\,g_{0}+(2M)^{-1}\right].
		\]
		By direct calculation, we find that $\bar{u}_{x}(0,x)\geq 0$ for $x\in [g_{0},\, g_{0}+(M)^{-1}]$. 
		Hence, 
		 \[
		 \bar{u}(0,x)\geq\bar{u}(0,\,g_{0}+(2M)^{-1})=\frac 34(2C_{1}-\delta)+\delta>C_{1}\geq u_{0}(x) \,\,\text{for}\,\,x \in[g_{0}+(2M)^{-1}, \,g_{0}+(M)^{-1}].
		 \]
		Combining the above results, we conclude that
		\[\bar{u}\left(0, x\right) \geq u_{0}(x)\text { for } x \in\left[g_{0},\, g_{0}+(M)^{-1}\right].\]
		Therefore, we can apply the standard comparison principle over the region $\Omega$
		to deduce that $u(t, x) \leq \bar{u}(t, x) $ in $\Omega$. Since $ u(t, g(t))=\bar{u}(t, g(t))=\delta $, it follows that
		\[u_{x}(t, g(t)) \leq \bar{u}_{x}(t, g(t))=2 M\left(2 C_{1}-\delta\right) \text { for } t \in\left[0, T_{1}\right).\]
		Consequently,
		\[ g^{\prime}(t)=-\frac{d}{\delta} u_{x}(t, g(t)) \geq -2 M\left(2 C_{1}-\delta\right) \frac{d}{\delta} \text { for } t \in\left(0, T_{1}\right],\]
		where  $M$ is chosen to be sufficiently large and independent of $T_{1}$.

	{\bf Step 4. } We show that for sufficiently large  $m$,  the inequality \eqref{2} holds for $t\in [0,T]$, and consequently, \eqref{2.26} and the first inequality in \eqref{1} are valid for all $t\in[0,T]$, thereby completing the proof of the lemma.

		We know that either $T_1=T$ or $T_1<T$, and by the definition of $T_1$, the latter case implies
		\begin{align}\label{2.27}
			g'(T_1)=dm=-\frac{d}{\delta} \underline u_x(T_1, g(T_1)).
		\end{align}
		Denote $w:=u-\underline u$. Then $w$ satisfies 
		\begin{equation*}
			\begin{cases}
				w_t> dw_{xx}+f(u)-f(\underline u)= dw_{xx}+c w,\ \ &t\in (0,T_1],\, x\in [g(t),g(t)+k/m],\\
				w(t,g(t))=0,\ w(t, g(t)+k/m)>0,\ & t\in [0,T_1],\\
				w(0,x)\geq  0,\ &x\in [g_0,g_0+k/m],
			\end{cases}
		\end{equation*}
		for some $c\in L^\infty$. By the comparison principle, 
		\begin{align*}
			w(t,x)>0\ \mbox{ for } \ t\in (0,T_1],\, x\in (g(t),g(t)+k/m].
		\end{align*} 
		Since $w(T_1, g(T_1))=0$, it follows from  the Hopf lemma that $w_x(T_1,g(T_1))>0$, and therefore
		\begin{align*}
			g'(T_1)=-\frac{d}{\delta}  u_x(T_1, g(T_1))<-\frac{d}{\delta}  \underline u_x(T_1, g(T_1))=dm,
		\end{align*}
		which contradicts  \eqref{2.27}. Therefore $T_1=T$
		and the proof of the lemma is completed.
	\end{proof}
	
	\begin{proof}[Proof of Theorem \ref{th1.1}]
		By Theorem \ref{th2.1}, we know that \eqref{1.1} has a unique solution \((u, g)\) defined on some maximal time interval \((0, T_{\max})\), and
		\[
		g \in C^{1+\frac{1+\alpha}{2}}\big((0, T_{\max})\big), \quad u \in C^{1+\frac{\alpha}{2}, 2+\alpha}(\Omega_{T_{\max}})
		\]
		with
		\[
		\Omega_{T_{\max}} := \big\{ (t, x) : t \in (0, T_{\max}), \, x \in [g(t), \infty) \big\}.
		\]
		Arguing indirectly, we assume \( T_{\max} < \infty \). 
		Hence, by Lemma \ref{le2.7}, we conclude that there exist constants \( C_1, C_2 > 0 \) such that for \( t \in [0, T_{\max}) \) and \( x \in [g(t),\infty) \),
		\begin{equation*}
			0 \leq u(x, t) \leq C_1, \quad |g'(t)| \leq C_2, \quad |g(t)| \leq C_2 t + |g_0|.
		\end{equation*}
		For any small constant \(\varepsilon > 0\), it follows from the proof of Theorem \ref{th2.1} that \( u \in C^{\frac{1+\alpha}{2}, 1+\alpha}(\overline{\Omega}_{T_{\max}-\varepsilon}) \). Therefore, as in  the proof of Theorem \ref{th2.1}, Schauder's estimates imply that for any fixed \( 0 < T < T_{\max} - \varepsilon \), we have
		\[
		\| u \|_{C^{1+\frac{\alpha}{2}, 2+\alpha}(\Omega_{T_{\max}-\varepsilon} \setminus \Omega_T)} \leq M,
		\]
		where \( M \) depends on \( T \), \( T_{\max}\), $C_1$ and $C_2$, but is independent of \( \varepsilon \). Since \( \varepsilon > 0 \) can be arbitrarily small, it follows that for any \( t \in [T, T_{\max}) \),
		\begin{equation*}
			\| u(t, \cdot) \|_{C^{2+\alpha}([g(t), \infty))} \leq M.
		\end{equation*}
		Now, repeating the proof of Theorem \ref{th2.1}, we conclude that there exists \( \varepsilon_0 > 0 \), depending only on \( M \) and \( C_1 \), $C_2$, such that \eqref{1.1} with initial time \( T_{\max} - \frac{\varepsilon_0}{2} \) has a unique solution \((u, g)\) defined for \( t \) up to \( T_{\max} + \frac{\varepsilon_0}{2} \). This contradicts the maximality of \( T_{\max} \). Thus, \( T_{\max} = \infty \).
	\end{proof}

	\section{Proof of Theorem \ref{th1.2}}
		
	\subsection{Behaviour of $g(t)$ and $u(t,x)$ for $t\gg 1$ and $x-g(t)\gg 1$}
	\begin{lemma}\label{u}
		
		Suppose $(\mathbf{f_m})$ holds, and  $(u, g)$ is the solution of \eqref{1.1} with  $u_0 \in \mathcal{X}(g_0)$. Then the following conclusions hold.
		\begin{enumerate}
			
			\item[(i)] There exists $T^0 > 0$ such that
			\[
			g'(t) > 0, \,	u(t,x) < \delta \quad \text{for } t > T^0, \, x \in (g(t), \infty),\]
			and  $\lim_{t \to \infty} g(t)=\infty$.
			\item[(ii)] For any  $\varepsilon > 0$, there exists $T_0(\varepsilon) > 0$, $X_0(\varepsilon)>0$ such that
			\begin{align}\label{3.1a}
				1+\frac{\varepsilon}{2}\geq u(t,x) \geq 1 - \frac{\varepsilon}{2} \quad\text{ for } t\geq T_0(\varepsilon) \text{ and } x\geq g(t)+X_0(\varepsilon).
			\end{align}
		\end{enumerate}
	\end{lemma}
	\begin{proof}
		
		Consider the following problems
		\begin{equation}\label{eq:ivp_v}
			\left\{
			\begin{aligned}
				v'(t) &= f(v), \quad t > 0, \\
				v(0) &= \inf_{x \geq g_0} u_0(x) \in (0, \delta],
			\end{aligned}
			\right.
		\end{equation}
		and
		\begin{equation}\label{eq:ivp_w}
			\left\{
			\begin{aligned}
				w'(t) &= f(v), \quad t > 0, \\
				w(0) &= \sup_{x \geq g_0} u_0(x)+1 > \delta.
			\end{aligned}
			\right.
		\end{equation}
		Since $f$ satisfies  $({\bf f_m})$, standard ODE theory implies $\lim_{t \to \infty} v(t) = \lim_{t \to \infty} w(t) = 1$. Consequently, there exists $T^0 > 0$ such that $w(T^0) = \delta$ and $w(t)>\delta$ for $t\in [0, T_0)$.
		Define $W(t,x) := w(t) - u(t,x)$. Then 
		\[
		\begin{cases} 
			W_t - d W_{xx} - cW = 0, & 0 < t \leq T^0, \quad x > g(t), \\ 
			W(t, g(t)) \geq  0, & 0 < t \leq T^0, \\ 
			W(0, x) \geq 0, & x \geq g_0, 
		\end{cases}
		\] 
		where $c(t,x) =  \dfrac{f(w(t)) - f(u(t,x))}{w(t) - u(t,x)} \cdot \mathbf{1}_{\{w(t)\neq u(t,x)\}}$ is bounded. Applying the maximum principle \cite[Theorem 2.7]{GML} yields $W(t,x) \geq 0$ for $x \in [g(t), \infty)$ and $t \in [0, T^0]$. In particular, $u(T^0, x) \leq w(T^0) = \delta$. We may then use the comparison principle to compare $\bar w\equiv \delta$ and the solution $u$ of \eqref{1.1} in the region $t \geq T^0$ and $x \in [g(t), \infty)$ to conclude that  $u(t,x) \leq \delta$ for $t \geq T^0$ and $x \in [g(t), \infty)$. The strong comparison principle further gives $u(t,x) < \delta$ for $t \geq T^0$ and $x \in (g(t), \infty)$.
		The Hopf boundary lemma \cite[Lemma 1.21]{WMX} then implies $ \partial_x u(t, g(t)) < 0$ for $t > T^0$. Combining this with the free boundary condition $g'(t) = -\frac{d}{\delta} \partial_x u(t, g(t))$ we obtain $g'(t) > 0$ for $t > T^0$.
		\medskip
		
		To complete the proof for (1), it remains to show
		 $g_\infty:=\lim_{t \to \infty}g(t)=\infty$. 
		 Otherwise $g_\infty<\infty$.
		We then define 
		\[
		y: = x-g(t), \quad V(t, y): =u(t, y+g(t)) \mbox{ for } t\geq 0,\ y\geq 0.
		\]
		Clearly $V$ satisfies 
		\begin{equation}\label{lz}
			\left\{
			\begin{array}{ll}
				V_{t} - g'(t) V_y - d V_{yy} = f(V), & t > 0,  \quad y > 0, \\
				V(t,0) = \delta, & t > 0, \\
				V_{y}(t, 0) = -\dfrac{\delta}{d} g'(t), & t > 0, \\
				V(0, y) = u_{0}(y + g_0), & y \geq 0.
			\end{array}
			\right.
		\end{equation}
		Lemma \ref{le2.7} implies that $|g'(t)|\leq C_2$, where $C_2$ is a constant independent of $t$.	Applying interior and boundary $L^p$ estimates to \eqref{lz}, we obtain that for any $p > 1$,
		\[
		\| V \|_{W_p^{1,2}([n, n+2] \times [L, L+1])} \leq C_p
		\]
		 for all integers $n \geq 0$ and all real numbers $L \geq 0$, where $C_p > 0$ is independent of $n$ and $L$. Taking $p$ sufficiently large and using the Sobolev embedding theorem, we deduce from the third equation in \eqref{lz} that $g'(t)$ is uniformly continuous for $t \geq 1$. This implies $g'(t) \to 0$ as $t \to \infty$.
		Let $\{t_n\}$ be an arbitrary sequence increasing to $\infty$ as $n \to \infty$, and define
		\[
		V_n(t, y) := V(t_n + t, y).
		\]
		Then,
		\begin{equation}\label{ln}
			\left\{
			\begin{array}{ll}
				(V_n)_{t} - g'(t+t_n) (V_n)_y - d (V_n)_{yy} = f(V_n), & t >-t_n, \quad  y > 0, \\
				V_n(t,0) = \delta, & t >-t_n, \\
				(V_n)_{y}(t, 0) = -\dfrac{\delta}{d} g'(t+t_n), & t >-t_n, \\
				V_n(-t_n, y) = u_{0}(y + g_0), & y \geq 0.
			\end{array}
			\right.
		\end{equation}
		Applying $L^p$ estimates to \eqref{ln}, using the Sobolev embedding theorem and a standard diagonal subsequence argument, we find that for some $\alpha \in (0,1)$, along a subsequence, $V_n \to \tilde{V}$ in $C_{\textup{loc}}^{(1+\alpha)/2,1+\alpha}(\mathbb{R} \times [0,\infty))$. The limit function $\tilde{V}$ satisfies
		\begin{equation}\label{lz2}
			\begin{cases}
				\tilde{V}_{t} - d \tilde{V}_{yy} = f(\tilde{V}), & t \in \mathbb{R},\quad y \in [0,\infty), \\
				\tilde{V}(t, 0) = \delta, & t \in \mathbb{R}, \\
				\tilde{V}_{y}(t, 0) = 0, & t \in \mathbb{R}.
			\end{cases}
		\end{equation}
		Since $\tilde{V}(t, y) \leq \delta$ for all $(t, y) \in \mathbb{R} \times [0,\infty)$ and   $\overline{V} \equiv \delta$ is a strict upper solution to \eqref{lz2}, the strong maximum principle yields
		\[
		\tilde{V}(t, y) <\delta \quad \text{for all } t \in \mathbb{R} \text{ and } y \in (0,\infty).
		\]
		Applying the Hopf boundary lemma at $y = 0$, we obtain
		\[
		\tilde{V}_y(t, 0) < 0 \quad \text{for all } t \in \mathbb{R},
		\]
		which contradicts the third equation $\tilde{V}_y(t, 0) = 0$ in \eqref{lz2}. This contradiction proves that $\lim_{t \to \infty}g(t)=\infty$. Part (i) of the lemma is now proved.

		Define $V(t,x) := u(t,x) - v(t)$. Then
		\begin{equation*}
			\begin{cases}
				V_t - d V_{xx} - \tilde{c}V = 0, & t > 0, \quad x > g(t), \\
				V(t, g(t)) > 0,                   & t > 0, \\
				V(0, x) \geq 0,                  & x \geq g_0,
			\end{cases}
		\end{equation*}
		where $\tilde{c}$ is bounded. Applying the maximum principle \cite[Theorem 2.7]{GML} gives $V(t,x) \geq 0$, i.e., $u(t,x) \geq v(t)$ for $t > 0$ and $x \in [g(t), \infty)$.
		Since $v(\infty)=1$, for any $\varepsilon > 0$, there exists $T(\varepsilon) > T^0$ large enough such that
		\[
		u(t,x) \geq v(t) > 1 - \frac{\varepsilon}{2} \quad \text{for } t > T(\varepsilon), \, x \in [g(t), \infty).
		\]
		
		Let $U := U_M(t,x;L)$ be the unique solution of the following problem
		\begin{equation}\label{L}
			\left\{
			\begin{array}{ll}
				U_{t} - d U_{xx} = f(U), & t > 0, \quad L < x < L + M, \\
				U(t, L) = U(t, L + M) = \delta, & t > 0, \\
				U(0, x) \equiv \delta, & L \leq x \leq L + M,
			\end{array}
			\right.
		\end{equation}
		where $M, L > 0$. Since $\underline{U} \equiv 1$ is a lower solution of \eqref{L} and $\overline{U} \equiv \delta$ is an upper solution of \eqref{L}, the comparison principle implies $1 \leq U(t, x) \leq \delta$ for $t \geq 0$ and $x \in [L, L + M]$. Moreover, $U(t, x)$ is nonincreasing in $t$ for $L \leq x \leq L + M$. It follows that
		\[
		v_{M}(x;L) := \lim_{t \rightarrow \infty} U_{M}(t, x;L) \text{ exists},
		\]
		and $1 \leq v_{M} \leq \delta$ for $x \in [L, L + M]$. Moreover, by standard parabolic regularity, the above limit holds in the $C^{2}([L, L + M])$ norm for $U_{M}(t, \cdot;L)$, and $v_{M}(x;L)$ satisfies
		\begin{equation}\label{stationary}
			-d v_{M}''(x) = f\left(v_{M}\right), \quad 1 \leq v_{M} \leq \delta \text{ in } (L, L + M), \quad v_{M}(L) = v_{M}(L + M) = \delta.
		\end{equation}

		{\bf Claim}. The function $v_M$ converges to $v^{*} := v^{*}(x;L)$ in $C_{\mathrm{loc}}^{2}([L, \infty))$  as $M\to \infty$ with $v^{*} $ satisfying
		\begin{equation}\label{v*}
			\left\{
			\begin{array}{ll}
				-d (v^{*})'' = f(v^{*}),\ (v^{*})' < 0\ \mbox{ for }  \ x > L, \\
				v^{*}(L) = \delta, \quad v^{*}(\infty) = 1. 
			\end{array}
			\right.
		\end{equation}
		
		For $M_{1} > M$, the function $U_{M_1}$ restricted to $[L, L + M]$ is a lower solution to \eqref{L}. It follows that $U_{M_1} \leq U_{M}$ for $t \geq 0$ and $x \in [L, L + M]$, implying $v_{M_1} \leq v_{M}$ over $[L, L + M]$. Therefore,
		\[
		v^{*}(x;L) := \lim_{M \rightarrow \infty} v_{M}(x;L) \quad \text{exists for every } x \geq L.
		\]
		Moreover, by standard elliptic regularity, the above limit holds in $C_{\mathrm{loc}}^{2}([L, \infty))$ and
		\[
		-d (v^{*})''(x;L) = f\left(v^{*}(x;L)\right),  \quad 1 \leq v^{*}(x;L) \leq \delta \quad \text{for } x \geq L, \ \quad v^{*}(L) = \delta.
		\]
		It then follows easily from elementary analysis that $v^{*} := v^{*}(x;L)$ satisfies \eqref{v*}.
		This completes the proof of the claim.

		From this claim, there exists \( z_0 > 0 \) independent of \( L \) such that \( v^{*}(L + x) < 1 + \frac{\varepsilon}{8} \) for \( x \geq z_0 \). Then, there exists \( M_0 > z_0 \) large enough such that \( v_{M_0}(L + z_0) < v^{*}(L + z_0) + \frac{\varepsilon}{8} < 1 + \frac{\varepsilon}{4} \). In addition, there exists \( \overline{T}(\varepsilon) > 0 \), independent of \( L \), such that
		\begin{align}\label{UM0}
			U_{M_0}(t, L + z_0; L) < 1 + \frac{\varepsilon}{2} \quad\mbox{for } t \geq \overline{T}(\varepsilon).
		\end{align}
		Now we are ready to give the upper bound of \( u \). Since \( g'(s)>0 \) for  \( s>T^0 \) and $g(\infty)=\infty$, we can enlarge \( T^0 \) such that \( g(t) = \max_{s \in [0, t]} g(s) \) for \( t \geq T^0 \).
		
		For any \( \widetilde{T} \geq \overline{T}(\varepsilon) \) and \( X_1 \geq g(\widetilde{T} + T^0) + z_0 \), we have \( L_{\widetilde T} := X_1 - z_0 \geq g(t) \) for \( t \in [T^0, \widetilde{T} + T^0] \). To simplify notations we write
		\begin{align*}
			U^L(t,x) = U_{M_0}(t,x;L) \quad\mbox{for } t \geq 0,\  x \in [L, L + M_0].
		\end{align*}
		Then \( U^{L_{\widetilde T} }(t,x) \) satisfies
		\begin{equation*}
			\left\{
			\begin{array}{ll}
				U^{L_{\widetilde T} }_{t} - d U^{L_{\widetilde T} }_{xx} - f(U^{L_{\widetilde T} }) = 0, & 0 < t \leq \widetilde{T},  L_{\widetilde T} < x < L_{\widetilde T} + M_0, \\
				U^{L_{\widetilde T} }(t, L_{\widetilde T}) = \delta > u(t + T^0, L_{\widetilde T}), & 0 < t \leq \widetilde{T}, \\
				U^{L_{\widetilde T} }(t, L_{\widetilde T} + M_0) = \delta > u(t + T^0, L_{\widetilde T} + M_0), & 0 < t \leq \widetilde{T}, \\
				U^{L_{\widetilde T} }(0, x) \equiv \delta \geq u(T^0, x), & L_{\widetilde T} \leq x \leq L_{\widetilde T} + M_0.
			\end{array}
			\right.
		\end{equation*}
		The comparison principle then gives 
		\[\mbox{ $u(t + T^0, x) \leq U^{L_{\widetilde T} }(t, x)$  for \( t \in [0, \widetilde{T}],\  x \in [L_{\widetilde T}, L_{\widetilde T} + M_0]\) .}
		\]
		In particular, \( u(\widetilde{T} + T^0, L_{\widetilde T} + z_0) \leq U^{L_{\widetilde T} }(\widetilde{T}, L_{\widetilde T} + z_0) \leq 1 + \frac{\varepsilon}{2} \) by \eqref{UM0}. 
		Noting that the value of \( L_{\widetilde{T}} + z_0 = X_1 \) can be chosen arbitrarily subject to \( X_1 \geq g(\widetilde{T} + T^0) + z_0 \), we obtain
		\begin{align*}
			u(\widetilde{T} + T^0, x)  \leq 1 + \frac{\varepsilon}{2} \quad\mbox{for } x \geq g(\widetilde{T} + T^0) + z_0.
		\end{align*}
		Then the arbitrariness of \( \widetilde{T} \geq \overline{T}(\varepsilon) \) implies
		\[
		u(t, x) \leq 1 + \frac{\varepsilon}{2} \quad \text{for } t \geq  \overline{T}(\varepsilon) + T^0 \text{ and } x \in [g(t) + z_0, \infty).
		\] 
		Hence the desired conclusion holds with \( X_0(\varepsilon) := z_0 \), \( T_0(\varepsilon) := \overline{T}(\varepsilon) + T^0 \).
		The proof is complete.	
	\end{proof}
	
	\subsection{Semi waves }
	
	We present several  results on the semi-waves of \eqref{1.1} and its perturbations, which will play a key role in the proof Theorem \ref{th1.2}. The detailed proofs of the results in this subsection are postponed to Section 4.
	
	\begin{theorem}\label{p1}
		Suppose that \( f \) satisfies \( (\mathbf{f_m}) \) and \( \delta > 1 \).
		Then for any \( c \in \mathbb{R} \), there exists a unique function \( q = q_c \in C^2([0, \infty)) \) satisfying
			\begin{equation}\label{eq:part1}
				\begin{cases}
					d q'' - c q' + f(q) = 0, & x > 0, \\
					q(x) > 1,\quad q'(x) < 0, & x \geq 0, \\
					q(0) = \delta, \quad q(\infty) = 1.
				\end{cases}
			\end{equation}
			Moreover, the following conclusions hold:
			\begin{enumerate}
			\item[(i)] \( c_1 < c_2 \) implies \( q_{c_1}(x) < q_{c_2}(x) \) in \( (0, \infty) \) and  \( q_{c_1}'(0) < q_{c_2}'(0) \), 
			\item[(ii)] $
				\lim_{c \to \bar{c}} q_c = q_{\bar{c}} \ \text{in } C^2_{\mathrm{loc}}([0, \infty))\cap L^\infty([0, \infty)),
			$
			\item[(iii)]  \( \xi(c) := q_c'(0) - \frac{\delta}{d} c \) is strictly decreasing for  \( c \in \mathbb{R} \),
			\item[(iv)] there exists a unique $c^{*}\in\R$ such that
						\begin{align*}\label{2.7a}
			\xi(c^*)=0,\ 	\xi(c) > 0 \quad \text{for } c < c^*, \quad \xi(c) < 0 \quad \text{for } c > c^*,
			\end{align*}
			\item[(v)] \( 0>c^* > P_0(\delta) := -\sqrt{2d \int_{\delta}^{1} f(s)\, ds} \), and the function \( \delta \mapsto c^*(\delta) \) is decreasing for \( \delta > 1 \), with
			\begin{align*}
				\lim_{\delta \to 1^+} c^*(\delta) = 0.
			\end{align*}
		\end{enumerate}
	\end{theorem}
	
	In our proof of  Theorem \ref{th1.2}, we will make use of the semi-waves of \eqref{1.1} with $f$ replaced by some suitable perturbations, in order to produce the desired upper and lower solutions. 
	
	For \( 0 < \xi < \delta \), we will consider  functions $h(u)$ with the following  general monostable property:
	\begin{equation}\label{gm}
	  \begin{cases}
		h \text{ is } C^{1}, \ h > 0 \text{ in } (0,\xi), \quad h < 0 \text{ in } (\xi, \infty), \\
		h(0) = h(\xi) = 0, \ h'(0) > 0 > h'(\xi).
	\end{cases}
	\end{equation}
	It is easy to see that Theorem \ref{p1} still holds when $f$ satisfies \eqref{gm} if we replace $1$ by $\xi$ in the statements of the theorem.

	For small $\varepsilon>0$, let the functions \( f_{\varepsilon}^1(s) \) and \( f_{\varepsilon}^2(s) \) satisfy  \eqref{gm} with $\xi=1-\varepsilon$ and $\xi=1+\varepsilon$, respectively, and moreover
	\[
	f_{\varepsilon}^1(s) < f(s) < f_{\varepsilon}^2(s) \quad \text{for } s > 0, \quad \lim_{\varepsilon \to 0} \| f_{\varepsilon}^i - f \|_{C^1([0, 2\delta])} = 0.
	\]
	
	Using Theorem \ref{p1}, we can prove the following results.
		\begin{corollary}\label{l2}
		Let the functions \( f_{\varepsilon}^i \), $i=1,2$,  be  as above.
		
		\begin{itemize}
			\item[{\rm (i)}]
			For any \( c \in \mathbb{R} \), the problem \eqref{eq:part1} with \( f \) replaced by \( f_{\varepsilon}^i \) admits a unique solution, denoted by \( q_c^i(x;\varepsilon) \), satisfying $	\lim_{x \to \infty} q_c^1(x;\varepsilon) = 1 - \varepsilon$ and $\lim_{x \to \infty} q_c^2(x;\varepsilon) = 1 + \varepsilon$, respectively.	Moreover,  $q_c^1(x;\varepsilon) < q_c^2(x;\varepsilon)$ for $x>0$, with $  [q_c^1(0;\varepsilon)]' < [q_c^2(0;\varepsilon)]'$ and 
			\[
			\lim_{\varepsilon \to 0} q_c^i(x;\varepsilon) = q_c(x) \quad \text{in } C^2_{\mathrm{loc}}([0, \infty)), \quad
			\lim_{\varepsilon \to 0} \| q_c^i(\cdot;\varepsilon) - q_c \|_{L^\infty([0, \infty))} = 0.
			\]
			\item[{\rm (ii)}]
			The problem \eqref{eq:part2} with \( f \) replaced by \( f_{\varepsilon}^i \) admits a unique solution, denoted by \( (c_i^*(\varepsilon), q_i^*(x;\varepsilon)) \), satisfying $	\lim_{x \to \infty} q_1^*(x;\varepsilon) = 1 - \varepsilon$ and $\lim_{x \to \infty} q_2^*(x;\varepsilon) = 1 + \varepsilon$, respectively.
			Moreover,
			\begin{align*}
				&c_1^*(\varepsilon) < c^* < c_2^*(\varepsilon), \quad 
				\lim_{\varepsilon \to 0} c_1^*(\varepsilon) = c^*, \quad 
				\lim_{\varepsilon \to 0} c_2^*(\varepsilon) = c^*, \\
				&\lim_{\varepsilon \to 0} q_i^*(\cdot;\varepsilon) = q_c(\cdot) \  \text{in } C^2_{\mathrm{loc}}([0, \infty))\cap L^\infty([0, \infty)).
			\end{align*}
		\end{itemize}
	\end{corollary}

	For \( \bar{c}_0 > c^* \), let \( \bar{q}_0 = \bar{q}_{\bar{c}_0} \) be the unique solution to 
	\begin{equation}\label{eq:initial}
		\begin{cases}
			q'' - c q' + f(q) = 0, & z \in (0, \infty), \\
			q(0) = \delta > 1,\quad q(\infty) = 1, \\
			q(z) > 1,\quad q'(z) < 0, & z \in [0, \infty),
		\end{cases}
	\end{equation}
	with $c=\bar c_0$.
	By Theorem \ref{p1}, we have \( \bar{c}_0 > \frac{d}{\delta} \bar{q}_0'(0) \), and thus there exists a sufficiently large constant \( M > 0 \) such that $\bar{c}_0 > \frac{d}{\delta} \bar{q}_0'(0) + \frac{1}{M}$.
	Then, for $n=0,1,2, ...$,  let \( \bar{q}_n= \bar{q}_{\bar{c}_n} \) be the solution of \eqref{eq:initial} with $c=\bar c_n$, and 
	\[
	 \bar{c}_{n+1}:=\frac{d}{\delta} \bar{q}_n'(0) + \frac{1}{M+n} , \quad n = 0, 1, 2, \ldots.
	\]
	
	Similarly, for \( \underline{c}_0 < c^* \), let \( \underline{q}_0=\underline{q}_{\underline{c}_0} \) be the unique solution of \eqref{eq:initial} with $c=\underline c_0$. Again, by Theorem \ref{p1}, one has \( \underline{c}_0 < \frac{d}{\delta} \underline{q}_0'(0) \), and hence, for sufficiently large \( M > 0 \), $\underline{c}_0 < \frac{d}{\delta} \underline{q}_0'(0) - \frac{1}{M}$.
	We may then define the sequence \( (\underline{c}_n, \underline{q}_n) := (\underline{c}_n, \underline{q}_{\underline{c}_n}) \) inductively with $\underline q_n$  the solution of \eqref{eq:initial} with $c=\underline c_n$ and
	\[
	 \underline{c}_{n+1}:=\frac{d}{\delta} \underline{q}_n'(0) - \frac{1}{M+n}, \quad n = 0, 1, 2, \ldots.
	\]
	
	\begin{lemma}\label{l4}
		Let \( (\bar{c}_n, \bar{q}_n) \) and \( (\underline{c}_n, \underline{q}_n) \) be defined as above. Then \( \underline{c}_n \nearrow c^* \) and \( \bar{c}_n \searrow c^* \) as \( n \to \infty \), and
		\begin{equation}\label{2.14a}
			\begin{cases}
				\displaystyle\lim_{n \to \infty} \bar{q}_n(x) = q^*(x), \quad \lim_{n \to \infty} \underline{q}_n(x) = q^*(x) \quad \text{in } C^2_{\mathrm{loc}}([0, \infty)),\\[1ex]
				\displaystyle\lim_{n \to \infty} \| \bar{q}_n - q^* \|_{L^\infty([0, \infty))} = 0, \quad 
				\lim_{n \to \infty} \| \underline{q}_n - q^* \|_{L^\infty([0, \infty))} = 0.
			\end{cases}
		\end{equation}
	\end{lemma}
	
The sequences \( (\bar{c}_n, \bar{q}_n) \) and \( (\underline{c}_n, \underline{q}_n) \)	will be used in our ``touching method" in the next subsection.

	\subsection{The touching method}   We introduce a new method to complete the proof of Theorem \ref{th1.2}. The method involves a family of upper and lower solutions obtained from some semi-waves (they are actually \( (\bar{c}_n, \bar{q}_n) \) and \( (\underline{c}_n, \underline{q}_n) \) in Lemma \ref{l4}), and the key idea is to obtain a family of well-controlled touches of the solution $u(t,x)$ of \eqref{1.1}  from above by a suitable upper solution in this family, and from below by a suitable lower solution in this family, which will provide the desired estimate for $g(t)$ and $u(t,x)$ by applying the usual comparison principle in a time interval immediately before the touching time.
	 
	  It is worth mentioning that the comparison principle developed in \cite{Du2024} (see Lemma 2.2 there) for moving boundary problems seems difficult to apply in our setting where the sem-wave is retreating.

	\begin{proof}[\underline {Proof of Theorem \ref{th1.2}}]
		For \( \varepsilon > 0 \), let the functions \( f_{\varepsilon}^1(u) \) and \( f_{\varepsilon}^2(u) \) satisfy \eqref{gm} with $\xi=1-\varepsilon$ and $\xi=1+\varepsilon$, respectively, and
		\begin{align*}
			&f_{\varepsilon}^1(s) < f(s) < f_{\varepsilon}^2(s) \quad \text{for} \quad s > 0,\\
			& \lim_{\epsilon \to 0} \| f_{\varepsilon}^i - f \|_{C^1([0, 2\delta])} = 0.
		\end{align*}
		Denote the corresponding semi-wave solutions by \( (c_{1}^{*}(\varepsilon), q_{1}^{*}(x; \varepsilon)) \) and \( (c_{2}^{*}(\varepsilon), q_{2}^{*}(x; \varepsilon)) \), which satisfy \eqref{eq:part2} with \( f \) replaced by \( f_{\varepsilon}^1 \) and \( f_{\varepsilon}^2 \), respectively. Then, by Corollary \ref{l2}, we have
		\begin{align}\label{2.26a}
			c_1^{*}(\varepsilon) < c^{*} < c_2^{*}(\varepsilon)<0, \quad \lim_{\varepsilon \to 0} c_1^{*}(\varepsilon) = c^{*}, \quad \lim_{\varepsilon \to 0} c_2^{*}(\varepsilon) = c^{*}.
		\end{align}
		
		Next, for each small positive \( \varepsilon \), we  show that
		\begin{align}\label{2.27a}
			0<-c_2^{*}(\varepsilon) \leq \liminf_{t \to \infty} g'(t) \leq \limsup_{t \to \infty} g'(t) \leq -c_1^{*}(\varepsilon).
		\end{align}
		We note that the desired conclusion for \( g'(t) \) follows from \eqref{2.27a} by letting \( \varepsilon \to 0 \) and using \eqref{2.26a}.
		
		From Lemma \ref{le2.7}, we know that \( 0\leq  \limsup_{t\to\infty}g'(t)\leq   c_M \) for some constant $c_M\geq 0$. Let 
		\[
		\begin{cases} \underline{c}_0 :=- \max\{c_M, -c_1^{*}(\varepsilon)\}-1<c^*,\\
		 \bar{c}_0 := 0 >c^*. 
		 \end{cases}
		 \]
		 Then we  define two sequences of semi-wave pairs \( (\underline{c}_n, \underline{q}_n) \) and \( (\bar{c}_n, \bar{q}_n) \) as in Lemma \ref{l4}, namely
		 \begin{align}\label{2.28c}
		\underline{c}_{n+1}:=	\frac{d}{\delta} \underline{q}_n'(0) - \frac{1}{M+n} < c_1^{*}(\varepsilon) < c_2^{*}(\varepsilon) < \bar{c}_{n+1}:=\frac{d}{\delta} \bar{q}_n'(0) + \frac{1}{M+n} \leq 0,
		\end{align}
		where $\underline q_n$  solves \eqref{eq:initial} with $(c, f)=(\underline c_n, f_{\varepsilon}^1)$,  $\bar q_n$ solves \eqref{eq:initial} with $(c, f)=(\bar c_n, f_{\varepsilon}^2)$, and 
		 \( M > 0 \) is a suitable constant. As in Lemma \ref{l4}, we can easily show
		 \begin{equation}\label{c_n}
		 \mbox{ \( \underline{c}_{n+1}\nearrow   c_1^{*}(\varepsilon) \) and \( \bar{c}_{n+1} \searrow c_2^{*}(\varepsilon) \) as \( n \to \infty \).}
		 \end{equation}
		
		In the following, we will use an induction argument to prove that for each \( n\geq 1\),
		\begin{align}\label{2.28}
			0\leq -\frac{d}{\delta} \bar{q}_n'(0)  \leq \liminf_{t \to \infty} g'(t) \leq \limsup_{t \to \infty} g'(t) \leq -  \frac{d}{\delta} \underline{q}_n'(0).
		\end{align}
		It is easily seen that \eqref{2.27a} follows from \eqref{2.28}.
		
		To prove \eqref{2.28}, we will construct a family of lower solutions associated with \( \underline{c}_n \), and a family of upper solutions associated with \( \bar{c}_n \). 
		Unlike the usual application of upper and lower solutions, here  the upper bound for \( |g'| \) will be obtained by using lower solutions.
		
		We will achieve this goal in four steps.

		\textbf{Step 1}: {Construction of the initial pair of upper and lower solutions.}
		
		With \( L > 0 \) regarded as a parameter in a range to be specified later,  we define a pair of lower and upper solutions, respectively, by
		\begin{align*}
			&\underline{u}^{L}(t,x) := \underline{q}_0(x -\underline{g}^{L}(t)), \quad t \geq 0,\ x \geq \underline{g}^{L}(t) := -\underline{c}_0 t - L,\\
			&\bar{u}^{L}(t,x) := \bar{q}_0(x -\bar{g}^{L}(t)), \quad t \geq 0,\ x \geq \bar{g}^{L}(t) := -\bar{c}_0 t + L,
		\end{align*}
		 Clearly, \( \underline{u}^{L}(t,x) \) is a semi-wave propagating  rightward with speed \( -\underline{c}_0 \), whose retreating front is \( x=\underline{g}^{L}(t) = -\underline{c}_0 t - L \); similarly, \( \bar{u}^{L}(t,x) \) moves  rightward with a slower speed \( -\bar{c}_0 \) and retreating front \( x=\bar{g}^{L}(t) = -\bar{c}_0 t + L \).

		By Lemma \ref{u}, there exists \( \tilde{T} \geq T_0 \) such that 
		\[
		\begin{cases}
		 1 - \frac{\varepsilon}{2} \leq u(t, x) \leq \delta & \mbox{ for  \( t \geq \tilde{T},\  x \in [g(t), \infty),\)}\\
		1 - \frac{\varepsilon}{2} \leq u(t, x) \leq 1 + \frac{\varepsilon}{2} &  \mbox{  for \( t \geq \tilde{T}, \ x \in [g(t)+X_0, \infty). \)}
		 \end{cases}
		 \]
	To simplify notations, in the following analysis, we assume \( \tilde{T} = 0 \), which can be achieved by replacing \( g_0 \) with \( g(\tilde{T}) \) and \( u_0(x) \) with \( u(\tilde{T}, x) \).

	Since $\underline q_0(+\infty)=1-\varepsilon$ and $\bar q_0(x)>1+\varepsilon$ for $x\geq 0$, we can choose a sufficiently large constant \( L_0 > |g_0| \) such that for all \( L \geq L_0 \),
		\begin{equation}\label{2.29a}
			\begin{cases}
				\underline{u}^{L}(0, x) < u(0, x),\quad &x \geq g_0,\\
				u(0,x) < \overline{u}^{L}(0,x), \quad &x \geq \bar{g}^{L}(0) = L.
			\end{cases}
		\end{equation}
		Based on this, we will later show that for each \( L \), the function \( \underline{u}^{L}(t, x) \) serves as a lower solution for \( u(t, x) \) up to the first contact time between their left boundaries \( \underline{g}^{L}(t) \) and \( g(t) \). This means that the graph of \( \underline{u}^{L} \) always lies strictly below that of the solution \( u \) before the two boundaries meet. A similar result holds for \( \overline{u}^{L} \), which acts as an upper solution to \( u \).

		Define
		\begin{equation*}
			\begin{cases}
				\underline T^0 := \sup \left\{ t \geq 0 :  \underline{g}^{L_0}(s) < g(s)\ \text{for all } s \in [0, t] \right\}, \\[2mm]
				\bar T^{0} := \sup \left\{ t \geq 0 :  \bar{g}^{L_0}(s) > g(s)\ \text{for all } s \in [0, t] \right\}.
			\end{cases}
		\end{equation*}
 Since $\underline g^{L_0}(0)=-\bar g^{L_0}(0)=-L_0<g(0)$, we see that $\underline T^0$ and $\bar T^0$ are well-defined. We claim that both $\underline T^0$ and $\bar T^0$ are finite.
		Indeed, since $\bar c_0=0$ and hence \( \bar{g}^{L_0}(s) \equiv L_0 \) is constant, and since \( \lim_{t \to \infty} g(t) = \infty \), we immediately see that \( \bar T^0 < +\infty \).  The assertion $\underline T^0<+\infty$ is a consequence of
		\[
		 \limsup_{t \to \infty} g'(t) \leq c_M < \max\{c_M, c_1^{*}(\varepsilon)\} + 1= -\underline{c}_0\equiv [\bar g^{L_0}(t)]'. 
		 \]
		  		
		\textbf{Step 2}. We use a touching method to prove that
		\begin{equation}\label{2.31c}
			\begin{cases}
				&  -\dfrac{d}{\delta} \, \bar{q}_0'(0)  \leq \liminf_{t \to \infty} g'(t) \leq \limsup_{t \to \infty} g'(t) \leq  - \dfrac{d}{\delta} \, \underline{q}_0'(0) ,\\
				&\underline{q}_0(x - g(t))<u(t, x)<	\bar{q}_0(x - g(t))\quad\mbox{for }  x > g(t),\  t\geq \max\{\underline T^0 , \bar T^0 \}.
			\end{cases}
		\end{equation}
		
		From the definition of \( \underline T^0\), we see that
		\[
		g(t)>\underline g^{L_0}(t) \mbox{ for } t\in [0, \underline T^0),\ g(\underline T^0)=\underline g^{L_0}(\underline T^0).
		\]
		It follows that 
		\[
		u(t, g(t))=\delta>\underline u(t, g(t)) \mbox{ for } t\in [0, \underline T^0).
		\]
		
	 We
		 now compare $u(t,x)$ with $\underline u^{L_0}(t,x)$ over the region $\Omega_0:=\{(t,x) \mid t \in [0, \underline T^0],\ x > g(t)\}$. Since \eqref{2.29a} holds and
	\[
	(\underline{u}^{L_0})_t - d (\underline{u}^{L_0})_{xx} = f_{\varepsilon}^1(\underline{u}^{L_0}) \leq f(\underline{u}^{L_0}) \mbox{ in } \Omega_0,
	\]
	The usual comparison principle gives
	\[
	\underline u^{L_0}(t,x)<u(t,x) \mbox{ for } (t,x)\in\Omega_0.
	\]
		Since 
	\[
	g(\underline T^0)=\underline g^{L_0}(\underline T^0) \mbox{ and }
	u(\underline T^0, g(\underline T^0))=\delta=\underline u^{L_0}(\underline T^0, \underline g^{L_0}(\underline T^0)),
	\]
	the Hopf boundary lemma implies	
	\[
		 u_x(\underline T^0, g(\underline T^0)) > \underline{u}_x^{L_0}(\underline T^0, g(\underline T^0)) =-\frac d\delta \underline{q}_0'(0) =:\beta_1,
	\]	
	and hence
	\[
	g'(\underline T^0)=-\frac d\delta u_x(\underline T^0, g(\underline T^0))<\beta_1.
	\]
	We have thus proved
		\begin{equation}\label{2.30}
			\begin{cases}
				 \underline{u}^{L_0}(t, x) < u(t, x) \mbox{ for } t\in [0, \underline T^0],\ x > g(t), \\
				\underline{g}^{L_0}(\underline T^0) = g(\underline T^0), \quad g'(\underline T^0) < -\dfrac{d}{\delta} \, \underline{q}_0'(0) =\beta_1.
			\end{cases}
		\end{equation}
		Therefore,	 in the $xu$-plane,  as time increases from $t=0$ to $t=\underline T^0$, the evolving graph of $u=u(t,x) (x\geq g(t))$ is touched by that of  the lower solution $u=\underline u^{L_0}(t,x) (x>\underline g^{L_0}(t))$ from below for the first time  at $t=\underline T^0$, and the touching point is $(x,u)=(g(\underline T^0), u(\underline  T^0, g(\underline T^0))=(g(\underline T^0), \delta) $. We will henceforth call the above argument which leads to \eqref{2.30} the ``touching method".

		This implies in particular that $g'(t)<\beta_1$ for $t>\underline T^0$ but close to $\underline T^0$. We show below, by the touching method again,  that
		\begin{equation}\label{2.32c}
			g'(t) < \beta_1 \quad \text{for all } \quad t>\underline T^0.
		\end{equation}
		
		Otherwise there exists \( T > \underline T^0 \) such that \( g'(t) < \beta_1 \) for \( t \in [\underline T^0, T) \) but \( g'(T) = \beta_1 \).
		 Define
		\[
		L^1 := -\underline{c}_0 T - g(T), \quad k^1(t) := \underline{g}^{L^1}(t) - g(t) = -\underline{c}_0 t - L^1 - g(t).
		\]
		We show next that 
		\begin{equation}\label{2.32}
			L^1 > L_0 ,\ k^1(t) < 0 \quad \text{for } t \in [0, T), \quad  k^1(T) = 0.
		\end{equation}
		Since $g(\underline T^0) = \underline{g}^{L_0}(\underline T^0) = -\underline{c}_0 \underline T^0 - L_0$ (see \eqref{2.30}),
		and by our assumption \( g'(t) <\beta_1 \) on \( [\underline T^0, T) \), we have
		\begin{align*}
			L^1 - L_0 &= -\underline{c}_0 T - g(T) + \underline{c}_0 \underline T^0 + g(\underline T^0) \\
			&= \int_{\underline T^0}^T \left[ -\underline{c}_0 + g'(s) \right] ds > 0.
		\end{align*}
		 Direct computation gives \( k^1(T) = -\underline{c}_0 T - L^1 - g(T) = 0 \).	For \( t \in [\underline T^0, T) \), we have $k^1(t)' = -\underline{c}_0 + g'(t) < 0$
		and hence \( k^1(t) < k^1(T)=0 \) for $t\in [\underline T^0, T)$. For \( t \in [0, \underline T^0] \),  \( L^1 > L_0 \) implies 
		\( \underline{g}^{L^1}(t) < \underline{g}^{L_0}(t) \), and thus $k^1(t) < \underline{g}^{L_0}(t) - g(t) \leq 0$ by
		the definition of \( \underline T^0 \). Hence \eqref{2.32} holds.
		
		We are now able to compare
		$ u(t,x) $ with $ \underline{u}^{L^1}(t,x)$ in the region $\Omega_1:=\{(t,x): t\in [0, T],\ x>g(t)\}$ and conclude that
		 \( u(t,x) > \underline{u}^{L_1}(t,x) \) in $\Omega_1$. Moreover, we can use the Hopf boundary lemma to obtain
		\[
		u_x(T, g(T)) > \underline{u}_x^{L^1}(T, g(T)) = \underline{q}_0'(0),
		\]
		which implies  
		\[
		g'(T) < -\frac{d}{\delta} \underline{q}_0'(0) = \beta_1,
		\]
		a contradiction to the assumption \( g'(T) = \beta_1 \). This proves 	
		 \eqref{2.32c}.
		 
		Denote for each \( s > \underline T^0 \),
		\[
		L_s:= -g(s) - \underline{c}_0 s.
		\]
		Since \( g'(s) <\beta_1< -\underline{c}_0 \) for all such \( s \), we see that \( L_s \) is strictly increasing, which implies that
		\[ \underline{g}^{L_s}(t) = g(t) \mbox{ for $ t = s $ and $\underline{g}^{L_s}(t) < g(t)$ for $t \in [\underline T^0, s)$}.
		\]
		By \eqref{2.30} and the monotonicity of $\underline q_0$ we obtain
		\[
		u(\underline T^0, x)> \underline q_0(x-\underline c_0 T_0-L_s) \mbox{ for } x> g(\underline T^0).
		\]		
		 Hence we can use the comparison principle to compare $u(t,x)$ with $\underline q_0(x -\underline c_0 t-L_s)$
		  over the region $\{(t,x): t\in [\underline T^0, s],\ x> g(t)\}$ to obtain $u(t,x)>\underline q_0(x -\underline c_0 t-L_s)$ in this region, and in particular,
		\begin{equation}\label{u>}
		\underline{q}_0(x - g(s))= \underline q_0(x -\underline c_0 s-L_s)< u(s, x) \quad \mbox{ for } x > g(s),\ s>\underline T^0.
		\end{equation}
		
		Analogously, $\bar T^0$ is the touching time of $u(t,x)$ by the upper solution $\bar u^{L_0}(t,x)$ from above, and we can similarly show, by using the comparison principle and Hopf boundary lemma,
		that for all \( t \geq \bar T^0 \),
		\begin{equation}\label{2.32b}
			g'(t) > -\bar{c}_0 \quad \text{and} \quad \bar{q}_0(x - g(t)) >u(t, x)\quad\mbox{for } x > g(t).
		\end{equation}
		
		From \eqref{2.32c}, \eqref{u>} and \eqref{2.32b}, we see that \eqref{2.31c} holds, and Step 2 is thus completed.
		\medskip
		
		{\bf Step 3}.  We use an induction argument to prove that for every positive integer $j$, 
		\begin{equation}\label{2.37a}
			\begin{cases}
				-\dfrac{d}{\delta} \, \bar{q}_j'(0) \leq \liminf_{t \to \infty} g'(t) \leq \limsup_{t \to \infty} g'(t) \leq - \dfrac{d}{\delta} \, \underline{q}_j'(0),\\
				\underline{q}_j(x - g(t)) < u(t, x) < \bar{q}_j(x - g(t))\quad\mbox{for   $x > g(t)$ and all sufficiently large \( t \)}.
			\end{cases}
		\end{equation}
		More precisely, we show that \eqref{2.37a}
		holds for \( j = n \) provided that it is valid for all \( 0 \leq j \leq n - 1 \).
		
		Define a family of lower and upper solutions as follows
		\begin{align*}
			\underline{u}_n^{L}(t,x) &:= \underline{q}_n(x - \underline{g}_n^{L}(t)), \quad t \geq 0,\ x \geq \underline{g}_n^{L}(t) := -\underline{c}_n t - L, \\
			\bar{u}_n^{L}(t,x) &:= \bar{q}_n(x - \bar{g}_n^{L}(t)), \quad t \geq 0,\ x \geq \bar{g}_n^{L}(t) := -\bar{c}_n t + L,
		\end{align*}
		where \( L > 0 \) is a large constant, and \( (\underline{c}_n, \underline{q}_n) \) and \( (\bar{c}_n, \bar{q}_n) \) are defined at the beginning of the proof by using \eqref{2.28c}.
		
		As in Step 1, choose a constant \( L_n > |g_0| \) such that for all \( L \geq L_n \),
		\[
		\begin{cases}
			\underline{u}_n^{L}(0, x) < u(0, x), & x \geq g_0, \\
			u(0, x) < \bar{u}_n^{L}(0, x), & x \geq \bar{g}_n^{L}(0) = L.
		\end{cases}
		\]
		
		Define 
		\begin{equation*}
			\begin{cases}
				\underline {T}^n:= \sup \left\{ t \geq 0 :  \underline g _n^{L_n}(s) < g(s) \ \text{for all } s \in [0, t] \right\}, \\
				\overline{T}^n := \sup \left\{ t \geq 0 : \bar g _n^{L_n}(s) > g(s) \ \text{for all } s \in [0, t]  \right\}.
			\end{cases}
		\end{equation*}
		Then \( \underline T^{n} < \infty \) and \( \bar T^n < \infty \), since by the induction assumption,
		\[
		-\dfrac{d}{\delta} \, \bar{q}_{n-1}'(0) \leq \liminf_{t \to \infty} g'(t) \leq \limsup_{t \to \infty} g'(t) \leq - \dfrac{d}{\delta} \, \underline{q}_{n-1}'(0),
		\]
		and 
		\[
		\begin{cases}
			[\underline{g}_n^{L_n}]' = -\underline{c}_n = -\dfrac{d}{\delta} \, \underline{q}_{n-1}'(0) + \dfrac{1}{M + n - 1} > \limsup_{t \to \infty} g'(t), \\
			[\bar{g}_n^{L_n}]' = -\bar{c}_n = -\dfrac{d}{\delta} \, \bar{q}_{n-1}'(0) - \dfrac{1}{M + n - 1} < \liminf_{t \to \infty} g'(t).
		\end{cases}
		\]
	
		Repeating the touching method in Step 2 for the pair \( (\underline{u}_n^{L_n}, \bar{u}_n^{L_n}) \), we see that \eqref{2.37a} holds for $j=n$.
		\medskip
		
		{\bf Step 4}. We complete the proof of the theorem.
		
		By Step 2 and Step 3, we see that \eqref{2.37a} holds for every positive integer $j$. 
		Now \eqref{1.2} follows directly from \eqref{2.37a} and \eqref{c_n}. To prove \eqref{1.3}, 
				it suffices to show that for any given small \( \epsilon_1 > 0 \), there exists \( T_{\epsilon_1} > 0 \) such that
		\begin{align}\label{3.21a}
			\sup_{x \geq g(t)} |u(t,x) - q^*(x - g(t))| < 2\epsilon_1 \quad \text{for all } t \geq T_{\epsilon_1}.
		\end{align}
		
		By Corollary~\ref{l2}, there exists small \( \varepsilon > 0 \) such that
		\[
		\sup_{x \geq 0} |q_i^*(x; \varepsilon) - q^*(x)| < \epsilon_1, \quad \text{for } i = 1, 2.
		\]
		Moreover, by Lemma~\ref{l4} (applied with \( f \) replaced by \( f_\varepsilon^i \)), there exists a large integer \( N > 0 \) such that
		\[
		\sup_{x \geq 0} |\underline{q}_N(x) - q_1^*(x; \varepsilon)| < \epsilon_1, \quad 
		\sup_{x \geq 0} |\bar{q}_N(x) - q_2^*(x; \varepsilon)| < \epsilon_1.
		\]
		In view of Step 3, there exists \( T_N > 0 \) such that 
		\[
		\underline{q}_N(x - g(t)) < u(t, x) < \bar{q}_N(x - g(t)) \quad \mbox{for } x > g(t),\  t \geq T_N.
		\]
		Combining the above estimates we obtain, for \( t \geq T_N \) and \( x > g(t) \), 
		\begin{align*}
			u(t,x) - q^*(x - g(t)) 
			&= \big[u(t,x) - \underline{q}_N(x - g(t))\big] 
			+ \big[\underline{q}_N(x - g(t)) - q_1^*(x - g(t); \varepsilon)\big] \\
			&\quad + \big(q_1^*(x - g(t); \varepsilon) - q^*(x - g(t))\big) \\
			&\geq 0 - \epsilon_1 - \epsilon_1 = -2\epsilon_1,
		\end{align*}
		\begin{align*}
			u(t,x) - q^*(x - g(t)) 
			&= \big[u(t,x) - \bar{q}_N(x - g(t))\big]
			+ \big[\bar{q}_N(x - g(t)) - q_2^*(x - g(t); \varepsilon)\big] \\
			&\quad + \big[q_2^*(x - g(t); \varepsilon) - q^*(x - g(t))\big] \\
			&\leq 0 + \epsilon_1 + \epsilon_1 = 2\epsilon_1.
		\end{align*}
		Hence \eqref{3.21a} holds, and \eqref{1.3} is proved. The proof is now complete.
	\end{proof}
	
	\section{Proof of the results on semi-waves in Subsection 3.2}\label{AppendixA}

	\subsection{An ODE for the  semi-wave}
	We call \( q(z) \) a semi-wave associated with \eqref{1.1} of speed \( c \), if \( (c, q(z)) \) satisfies
	
	\begin{equation}\label{s1}
		\left\{
		\begin{array}{l}
			dq'' - c q' + f(q) = 0,\  q>1 \quad \text{for} \quad z \in (0, \infty), \\
			q(0) = \delta>1, \quad q(\infty) = 1.
		\end{array}
		\right.
	\end{equation}
	The first equation in \eqref{s1} can be written in the equivalent form
	\begin{equation}\label{s2}
		q' = p, \quad p' = \frac{1}{d}[c p - f(q)].
	\end{equation}
	So a strict decreasing solution \( q(z) \) of \eqref{s1} corresponds to a trajectory \( (q(z), p(z)) \) of \eqref{s2} that starts from the point \( (\delta, \omega) \) with \( \omega  < 0 \) in the \( (q, p) \)-plane and ends at the point \( (1, 0) \) as \( z \to \infty \).
	
	If \( p(z) = q'(z) < 0 \) for all \( z > 0 \), then the trajectory can be expressed as a function \( P(q) \) for \( q \in [1,  \delta] \) satisfying
		\begin{align}\label{P}
			\frac{dP}{dq} \equiv P' = \frac{c}{d} - \frac{f(q)}{dP} \quad \text{for } q \in (1, \delta), \quad   P(1) = 0, \ P(\delta) = \omega.
	\end{align}
	To emphasize the dependence on \( c \), we denote \( P(q) \) by \( P_c(q) \), and we will investigate the properties of this function.	It is easily checked that
	\[
	P_0(q) := -\sqrt{\frac{2}{d} \int_{q}^{1} f(s)  ds}, \quad q \in [1, \delta], 
	\]
	solves \eqref{P} with \( c = 0 \) and \( \omega = \omega^{0} :=- \sqrt{\frac{2}{d} \int_{\delta}^{1} f(s)  ds} < 0 \). Moreover,
	\[
	P_0^{\prime}(1) = -\sqrt{-\frac{f^{\prime}(1)}{d}}.
	\]
	
	Consider the equilibrium point \( (1,0) \) of \eqref{s2}. A simple calculation shows that \( (1,0) \) is a saddle point, and hence by the theory of ODE (cf. \cite{n1}) there are exactly two trajectories of \eqref{s2} that approach \( (1,0) \) from \( q > 1 \); one of them, denoted by \( T_{c} \), has slope
	\[
	\frac{c - \sqrt{c^{2} - 4d f^{\prime}(1)}}{2d} < 0
	\]
	at \( (1,0) \), and the other has slope
	\[
	\frac{c + \sqrt{c^{2} - 4 df^{\prime}(1)}}{2d} > 0
	\]
	at \( (1,0) \). A part of \( T_{c} \) that lies in the set
	\[
	S := \{(q, p):   q \geq 1,  p \leq 0\}
	\]
	and contains \( (1,0) \) is a curve which can be expressed as \( p = P_{c}(q) \) for \( q \geq 1\) in a small right neighburhood of $1$. Thus \( P_{c}(q) \) satisfies
	\begin{equation}\label{s3a}
	P' = \frac{c}{d} - \frac{f(q)}{dP} \quad \mbox{ for } q>1, \quad P(1) = 0, \quad P^{\prime}(1) = \frac{c - \sqrt{c^{2} - 4 df^{\prime}(1)}}{2d}<0.
	\end{equation}

	{\bf Claim:} \( P_c(q) < 0 \) for all \( q \in (1, \delta] \). 
	
	Suppose, for contradiction, that there exists some \( q_0 \in (1, \delta] \) such that \( P_c(q) < 0 \) for all \( q \in (1, q_0) \), but \( P_c(q_0) = 0 \). 
	  Given \( f(q_0) < 0 \), and because \( P_c(q) \to 0^- \) as \( q \to q_0^- \), it follows that
	$\lim_{q \to q_0^-} \frac{f(q)}{P_c(q)} = \infty$.
	Consequently,  we have
	\[
	\lim_{q \to q_0^-} P_c'(q) = \frac{c}{d} - \lim_{q \to q_0^-} \frac{f(q)}{d P_c(q)} = \frac{c}{d} - \infty = -\infty,
	\]
	which implies $P_c(q_0)<P_c(q_0-\epsilon)<0$ for all small $\epsilon>0$. This contradiction proves the claim.

	\subsection{Some perturbed semi-waves} 	Given \( c \in \mathbb{R} \) and a function \( h \) satisfying \eqref{gm}, we consider the  problem
	\begin{align}\label{s6}
		\frac{dP}{dq} = P' = \frac{c}{d} - \frac{h(q)}{dP} \quad \mbox{for } q >\xi, \quad P(\xi) = 0, \quad P'(\xi) = \frac{c - \sqrt{c^2 - 4d f'(\xi)}}{2d}<0.
	\end{align}
	Similar to the analysis in  Subsection 4.1 above, we can show that $P(q)$ is defined for $q\in [\xi, \delta]$ and $P(q)<0$ for $q\in (\xi, \delta]$.	
	
	\begin{lemma}  \label{l1}
		Let \( P_c(q)=P_c(q;h) \) be the solution to \eqref{s6} with \( h \) satisfying \eqref{gm}.  Then the following conclusions hold.
		\begin{enumerate}
			\item[(i)] For any \( c_1 < c_2 \),  
			\[
			P_{c_1}(q) < P_{c_2}(q)<0 \quad \text{for all } q \in (\xi, \delta];
			\]
			moreover, for any \( \bar{c} \in \mathbb{R} \),  
			\[
			\lim_{c \to \bar{c}} P_c(q) = P_{\bar{c}}(q) \quad \text{uniformly on } [\xi, \delta].
			\]
			
			\item[(ii)] Fix \( c\in \mathbb{R}\), and let \( f_1, f_2 \) satisfy \eqref{gm} with  \( \xi=m_1 \) and \( \xi=m_2 \) respectively, where \( 0 < m_1,  m_2 < \delta \). If \( f_1(u) < f_2(u) \) for all \( u \in (0, \delta) \), then  
			\[
		m_1<m_2,\ 	P_c(q;f_2) < P_c(q;f_1) \quad \text{for all } q \in [m_2, \delta];
			\]
			moreover, if \( f_n \to f \) in \( C^1([0, \delta]) \) as $n\to\infty$, where \( f_n \) and \( f \) satisfy \eqref{gm} with  \( \xi=m_n \) and \( \xi=m \), respectively, then
			\[
			\lim_{n\to\infty} m_n=m,\ \lim_{n \to \infty} P_c(q; f_n) = P_c(q; f) \quad \text{locally uniformly in } (m, \delta].
			\]
		\end{enumerate}
	\end{lemma}

	\begin{proof} (i) When \( c_1 < c_2 \), since
		\[
		P_c'(\xi) = \frac{c - \sqrt{c^2 - 4d h'(\xi)}}{2d}=\frac{2h'(\xi)}{c + \sqrt{c^2 - 4d h'(\xi)}} \mbox{ is always increasing in } c,
		\]
		we obtain \( P_{c_1}'(\xi) < P_{c_2}'(\xi) \).
		From \eqref{s6} and the condition $c_1 < c_2$, we obtain
		\begin{equation}\label{b1}
			P_{c_1}'(q) < \frac{c_2}{d} - \frac{h(q)}{d P_{c_1}(q)}, \quad q \in (\xi, \delta].
		\end{equation}
		This implies that, as $q$ increases from $\xi$, the curve $p = P_{c_1}(q)$ remains below $p = P_{c_2}(q)$ in the $(q, p)$-plane.  
		To verify this, suppose for contradiction that there exists a first point $q^0 \in (\xi, \delta]$ such that $P_{c_1}(q^0) = P_{c_2}(q^0)$ and $P_{c_1}(q) < P_{c_2}(q)$ for all $q \in (\xi, q^0)$. Then necessarily $P_{c_1}'(q^0) \geq P_{c_2}'(q^0)$. However, \eqref{b1} implies that at $q = q^0$,
		\[
		P_{c_1}'(q^0) < \frac{c_2}{d} - \frac{h(q^0)}{d P_{c_1}(q^0)} = P_{c_2}'(q^0).
		\]
		This contradiction shows that $P_{c_1}(q)<P_{c_2}(q)$ for $q\in (\xi,\delta]$. Thus $P_c(q)$ is strictly increasing in $c$ for each fixed $q \in (\xi, \delta]$.
		
		For any $\bar{c} \in \mathbb{R}$, as $c \uparrow \bar{c}$, the monotonicity implies that  $P_{c}(q)$ converges monotonically to some function 
		$R(q)$  in $[\xi, \delta]$. For $c \leq  \bar{c} $,  the inequality $P_c(q) \leq  P_{\bar{c}}(q)<0$ for $q\in (\xi,\delta]$ gives  
		\[
		0> \frac{h(q)}{d P_c(q)}  \geq  \frac{h(q)}{d P_{\bar{c}}(q)} \geq M_0:=\inf_{q\in [\xi,1]} \frac{h(q)}{d  P_{\bar{c}}(q)}\in (-\infty, 0).
		\]
		This implies, by \eqref{s6},  that \( \|P_c\|_{C^1([\xi, \delta])} \leq M_1 \) for all $c\leq \bar c$ and some positive constant $M_1$. By the Arzelà--Ascoli theorem, as \( c \to \bar{c} \), a subsequence of \( \{P_c\}_{n=1}^{\infty} \), still denoted by \( \{P_c\} \), converges to some function  in \( C^1([\xi, \delta]) \). This function necessarily coincides with $R(q)$. Hence the convergence holds as $c\to\bar c$ and
		 \( R(q) \) satisfies \eqref{s6} with \( R(\xi) = 0 \) and 
		$R'(\xi) = \frac{\bar{c} - \sqrt{\bar{c}^2 - 4d h'(\xi)}}{2d} < 0$.
		By the uniqueness of the solution to \eqref{s6}, we conclude that \( R(q) \equiv P_{\bar{c}}(q) \). By a similar argument, the same conclusion holds when \( c \downarrow \bar{c} \). The proof of part (i) is now complete.
			
		(ii)	Since $f_1(u)<f_2(u)$ for $u\in (0, \delta]$, we have $m_1<m_2$. Moreover, by (i),
		\begin{equation}\label{s3}
			P_{c}^{\prime}(q;f_1) < \frac{c}{d} - \frac{f_2(q)}{dP_{c}(q;f_1)} \mbox{ for } q\in (m_2, \delta], \quad P_{c}(m_2;f_1) <0= P_{c}(m_2;f_2).
		\end{equation}
		By the comparison argument used in part (i) above, we obtain \(P_{c}(q;f_1) < P_{c}(q;f_2)\) for all \(q \in [m_2, \delta]\). 
				
		Now, suppose $f_n\to f$ in $C^1([0,\delta])$ as \(n \to \infty\) with \(f_n\leq f\). Then, $m_n\leq m$, $m_n\to m$ and  \(P_c^n(q) := P_c(q;f_n)\leq  P_c(q;f)\) holds for all \(q \in [m, \delta]\) and \(n \geq 1\). 
		As in the proof of (i), we can show that a subsequence of \( P_c^n(q) \), still denoted by \( P_c^n(q) \), converges to some function \( S(q) \) in \( C^1([m, \delta]) \) as \( n \to \infty \). It then follows that \( S(q) \) satisfies  \eqref{s6} with \( S(m) = 0 \) and $S(q)\leq P_c(q;f)$ for $q\in [m,\delta]$. Thus, 
		\begin{align*}
			S'(m) \leq  P_c'(q;f)=\frac{c - \sqrt{c^2 - 4d f'(m)}}{2d}<0.
		\end{align*}
		As there is only one such solution, \(S(q)\) must coincide with \(P_c(q;f)\). This implies that \(P_c^n(q) \to P_c(q;f)\) as \(n \to \infty\) locally uniformly in \((m, \delta]\).
		
		Similarly, when  $f_n\to f$ in $C^1([0,\delta])$ as \(n \to \infty\) with \(f_n\geq f\), we obtain \(P_c^n(q) \to P_c(q;f)\) locally uniformly in \((m, \delta]\). 
		
		Now for a general sequence $f_n$ satisfying $f_n\to f$ in $C^1([0,\delta])$ as \(n \to \infty\), we can find $\delta_0>0$ small such that $m_n>2\delta_0$
		for all large $n$, say $n\geq n_0$. Then, by passing to a subsequence of $f_n$, we can find $\epsilon_n>0$, $\epsilon_n\to 0$ as $n\to\infty$ such that
		\[
		\bar f_n:=f_n+\epsilon_n\geq f,\ \underline f_n:=f_n-\epsilon_n\leq f,\ 2\delta_0<\underline m_n<m_n<\bar m_n \mbox{ for $n\geq n_0$ in $[2\delta_0, \delta]$},
		\]
		where $\underline m_n, \overline m_n$ are the unique zeros of $\underline f_n$ and $\bar f_n$, respectively. Our earlier analysis implies
		\[
		P_c(q; \underline f_n)\leq P_c(q; f_n)\leq P_c(q;\bar f_n) \mbox{ for } n\geq n_0,\ q\in [\bar m_n, \delta],
		\]
		and as $n\to\infty$,
		\[
		P_c(q; \underline f_n)\to P_c(q; f),\ P_c(q; \bar f_n)\to P_c(q; f) \mbox{ locally uniformly in } (m, \delta].
		\]
		It follows that $P_c(q;  f_n)\to P_c(q; f) \mbox{ locally uniformly in } (m, \delta]$ as $n\to\infty$. This completes the proof.
	\end{proof}

	\subsection{Proof of Theorem~\ref{p1}}
		 We first prove the existence and uniqueness of the semi-wave \( q_c \), as well as its monotonicity with respect to \( c \).
		
		Consider the following initial value problem
		\begin{equation}\label{q(x)}
				q' = p(q) \mbox{ for } x > 0, \
				q(0) = \delta,
		\end{equation}
		where \( p(q) \in C^1([1, \delta]) \) is the solution of  \eqref{s3a}, which is known to satisfy $p(q)<0$ for $q\in (1, \delta]$ and $p(1)=0$. It follows easily that
		$q(x)$ is defined for all $x>0$, and
		\[
		  q'(x)<0,\ 1<q(x)\leq \delta \mbox{  over } [0, \infty) \mbox{ with } q(x)\to 1 \mbox{ as } x\to\infty.
		\]
		For \( x >0 \), differentiating \( q'(x) = p(q(x)) \) yields
		\[
		q''(x) = \frac{d p}{d q}(q(x)) \cdot q'(x) = \frac{d p}{d q}(q(x)) \cdot p(q(x)) = \frac{1}{d} \left[ c p(q(x)) - f(q(x)) \right],
		\]
		which implies
		\begin{equation*}\label{ss1}
			d q'' - c q' + f(q) = 0.
		\end{equation*}
		Thus, \( q \) satisfies \eqref{eq:part1}.
		
		To show uniqueness, let $q(x)$ be any solution of \eqref{eq:part1}. Then the analysis in subsection 4.1 shows that $q'=p=P_c(q)$. Thus $q$ solves \eqref{q(x)}, and by the 
uniqueness of the solution to the initial value problem \eqref{q(x)}, we conclude that \eqref{eq:part1} has a unique solution.
		
		If \(c_1 < c_2\), then Lemma \ref{l1} implies \(P_{c_1}(q) < P_{c_2}(q)\) for all \(q \in (1, \delta]\). 
	Thus \(q'_{c_1}(0) = P_{c_1}(\delta) < q'_{c_2}(0) = P_{c_2}(\delta)\). 
		By continuity, \(q_{c_1}(x) < q_{c_2}(x)\) for \(x \in (0, \varepsilon_0)\) for some \(\varepsilon_0 > 0\)  sufficiently small. 
		Suppose there exists \(x_0 > 0\) such that \(q_{c_1}(x) < q_{c_2}(x)\) for \(x \in (0, x_0)\) and \(q_{c_1}(x_0) = q_{c_2}(x_0)\). 
		Then \(q'_{c_1}(x_0) \geq q'_{c_2}(x_0)\). 
		However, \(q'_{c_1}(x_0) = P_{c_1}(q_{c_1}(x_0)) < P_{c_2}(q_{c_2}(x_0)) = q'_{c_2}(x_0)\), a contradiction. 
		Thus \(q_{c_1}(x) < q_{c_2}(x)\) for all \(x \in (0, \infty)\).
		
		For any $\bar{c} \in \mathbb{R}$, as $c \uparrow \bar{c}$, the monotonicity implies that $q_c(x)$ is nondecreasing in $c$ for $x \in [0, \infty)$. Given $1 \leq q_c(x) \leq \delta$, there exists a limit function $Q(x)$ with $1 \leq Q(x) \leq \delta$ such that $q_c(x) \to Q(x)$ as $c\rightarrow \bar{c}$ for $x\in [0, \infty)$.  
		Applying the Sobolev embedding theorem and $L^p$ interior estimates to \eqref{eq:part1}, for large $p > 1$ and $\alpha \in (0,1)$, we obtain
		\[
		\|q_c\|_{C^{1+\alpha}([M, M+1])} \leq C_1 \|q_c\|_{W^{2,p}([M,M+1])} \leq C_2,
		\]
		with constants $C_1, C_2$ independent of $c$ and $M$. Combining with condition (\textbf{f}) and Schauder estimates, we get $\|q_c\|_{C^{2+\alpha}([0, M])} \leq C(M)$, where $C(M)$ depends on $M$ but not $c$. Hence  $q_c \to Q$ in $C^{2}_{\mathrm{loc}}([0, \infty))$ as $c\uparrow \bar{c}$	and $Q$ satisfies
		\begin{equation*}
			\begin{cases}
				d Q'' - \bar{c} Q' + f(Q) = 0, & x > 0, \\
				Q(x) > 1,\quad Q'(x) < 0, & x \in [0, \infty), \\
				Q(0) = \delta, \quad
				Q(\infty) = 1.
			\end{cases}
		\end{equation*}
		By the uniqueness of the solution to \eqref{eq:part1} with $c=\bar c$, we conclude $Q = q_{\bar{c}}$. The same conclusion holds for $c \downarrow \bar{c}$.  	
		Thus, \( q_c \to Q \) in \( C^2_{\mathrm{loc}}([0, \infty)) \) as \( c \to \bar{c} \). This convergence, together with the monotonicity of \( q_c \) with respect to \( c \) and the fact that \( q_c(\infty) = 1 \), implies
		\begin{align*}
			\lim_{c \to \bar{c}} \| q_c - q_{\bar{c}} \|_{L^\infty([0, \infty))} = 0.
		\end{align*}

		We next prove the existence and uniqueness of the semi-wave pair $(c, q_c)$ satisfying \( q_c'(0) = \dfrac{c\delta}{d} \).
		
		By \eqref{s3a}, we have $	P_c(\delta) = \frac{c}{d}(\delta - 1) - \int_{1}^{\delta} \frac{f(s)}{d \, P_c(s)} \, ds$.
		Define 
		\[
		\xi(c) := P_c(\delta) - \frac{\delta}{d} c = q_c'(0) - \frac {\delta}{d}c.
		\]
		Substituting the expression for \( P_c(\delta) \) into this definition yields 
		\begin{align*}
			\xi(c) = -\frac{c}{d}  - \int_{1}^{\delta} \frac{f(s)}{d \, P_c(s)} \, ds.
		\end{align*}
		This expression indicates that $\xi(c)$ is strictly decreasing for \( c \in \mathbb{R} \), since  \( f(s)<0 \) on \( (1, \delta) \) and  \( P_c(s) < 0 \) is increasing in \( c \) (see Lemma~\ref{l1}).
		
	Clearly, \( \xi(0) = -\int_{1}^{\delta} \frac{f(s)}{d \, P_0(s)} \, ds < 0 \), and with \( c_0 := d \, P_0(\delta) < 0 \), we have
		\[
		\xi(c_0) = -P_0(\delta) - \int_{1}^{\delta} \frac{f(s)}{d \, P_{c_0}(s)} \, ds > -P_0(\delta) - \int_{1}^{\delta} \frac{f(s)}{d \, P_0(s)} \, ds = 0,
		\]
		where the last equality follows from \eqref{s3a} with $c=0$.
		Since \( \xi(c) \) is continuous and strictly decreasing, from \( \xi(c_0) > 0 > \xi(0) \) we immediately see that there exists a unique \( c^* \in (c_0, 0) \) such that \( \xi(c^*) = 0 \).

		Next, we establish the monotonicity of \( c^*(\delta) \) with respect to \( \delta \). To emphasize the dependence of \( \xi(c) \) on \( \delta \), we denote it by \( \xi^\delta(c) \). Since both \( f \) and \( P_c \) are negative on \( (1, \delta) \), it follows from the definition of \( \xi^\delta(c) \) that \( \xi^\delta(c) \) is strictly decreasing with respect to \( \delta \). Suppose \( 1 < \delta_1 < \delta_2 \), and let \( c^*(\delta_i) \) be the unique value satisfying \( \xi^{\delta_i}(c^*(\delta_i)) = 0 \) for \( i = 1, 2 \). Then,
		\[
		\xi^{\delta_2}(c^*(\delta_1)) < \xi^{\delta_1}(c^*(\delta_1)) = 0.
		\]
		Since \( \xi^{\delta_2}(c) \) is strictly decreasing in \( c \) and satisfies \( \xi^{\delta_2}(c^*(\delta_2)) = 0 \), we deduce that \( c^*(\delta_1) > c^*(\delta_2) \). Therefore, \( c^*(\delta) \) is strictly decreasing for \( \delta > 1 \).

		Since \( c^*(\delta) < 0 \) for all \( \delta > 1 \), we see that \( \lim_{\delta \to 1^+} c^*(\delta) =: c^0 \) exists with \( c^0 \leq 0 \). 
		For any fixed \( c < 0 \), we observe that
		\[
		\lim_{\delta \to 1^+}\frac{f(s)}{d P_c(s)}=\frac{f'(1)}{dP_c'(1)}\in (0,\infty),\ \ \lim_{\delta \to 1^+} \int_{1}^{\delta} \frac{f(s)}{d P_c(s)} \, ds = 0,\]
		 and hence
		 \[
		   \lim_{\delta \to 1^+} \xi^\delta(c) = \lim_{\delta \to 1^+} \left( -\frac{c}{d}  - \int_{1}^{\delta} \frac{f(s)}{d P_c(s)} \, ds \right) = -\frac{c}{d}  < 0.
		\]
		In particular, it follows that \( \lim_{\delta \to 1^+} \xi^\delta(c^0) = -\frac{c^0}{d} \). If \( c^0 < 0 \), then by the continuity of \( \xi^\delta(c) \) in both \( \delta \) and \( c \), there exists \( \varepsilon > 0 \) such that 
		\[
		\xi^\delta(c) < -\frac{c^0}{2d} < 0 \quad \text{for all } \delta \in (1, 1+\varepsilon) \text{ and } c \in [c^0 - \varepsilon, c^0 + \varepsilon].
		\]
		However, this contradicts with \( \xi^\delta(c^*(\delta)) = 0 \) and \( \lim_{\delta \to 1^+} c^*(\delta) = c^0 \).  Therefore, we must have \( c^0 = 0 \).
		
		Now all the conclusions in the theorem are proved, and  the proof is thus completed. \qed

	\subsection{Proof of Corollary \ref{l2}}  
		(i)  Theorem \ref{p1} implies that \eqref{eq:part1} with \( f \) replaced by \( f_{\varepsilon}^i \), $i=1,2$, admits a unique solution, denoted by \( q_c^i(x;\varepsilon) \), satisfying \[
		\mbox{\(\lim_{x \to \infty} q_c^1(x;\varepsilon) = 1 - \varepsilon\),\ \(\lim_{x \to \infty} q_c^2(x;\varepsilon) = 1 + \varepsilon\). }
		\]
		
		Since \( f_{\varepsilon}^1(s) < f(s) < f_{\varepsilon}^2(s) \) for \( s > 0 \), Lemma \ref{l1} implies \( P_{c}(q; f_\varepsilon^1) < P_{c}(q; f_\varepsilon^2) \) for all \( q \in [1+\varepsilon, \delta]\).  
		It follows that \((q_c^1)'(0;\varepsilon) = P_{c}(\delta; f_\varepsilon^1) < (q_c^2)'(0;\varepsilon) = P_{c}(\delta; f_\varepsilon^2) \), and hence 
		 \( q_c^1(x;\varepsilon) < q_c^2(x;\varepsilon) \) for \( x \in (0, \varepsilon^0) \) with \( \varepsilon^0 > 0 \)  sufficiently small.  
		Suppose there exists \( x^0 > 0 \) such that \( q_c^1(x;\varepsilon) < q_c^2(x;\varepsilon) \) for \( x \in (0, x^0) \) and \( q_c^1(x^0;\varepsilon) = q_c^2(x^0;\varepsilon) \).  
		Then \( (q_c^1)'(x^0;\varepsilon) \geq (q_c^2)'(x^0;\varepsilon) \).  
		However, 
		\[
		(q_c^1)'(x^0;\varepsilon) = P_{c}(q_c^1(x^0;\varepsilon); f_\varepsilon^1) < P_{c}(q_c^1(x^0;\varepsilon); f_\varepsilon^2) = (q_c^2)'(x^0;\varepsilon),
		\] 
		a contradiction.  
		Thus, \( q_c^1(x;\varepsilon) < q_c^2(x;\varepsilon) \) for all \( x \in (0, \infty) \).  
		Similarly, we can prove that \( q_c^1(x;\varepsilon) \) is nondecreasing in \( \varepsilon \) and \( q_c^2(x;\varepsilon) \) is nonincreasing in \( \varepsilon \) for \( x \in [0, \infty) \) with \( q_c^1(x;\varepsilon) <q_c(x)< q_c^2(x;\varepsilon) \) for all \( x \in (0, \infty) \). 
		
		Since \( 1-\varepsilon \leq q_c^1(x;\varepsilon) \leq \delta \), the above monotonicity property implies the existence of a function \( Q_1(x) \) with \( 1 \leq Q_1(x) \leq \delta \) such that \( q_c^1(x;\varepsilon) \to Q_1(x) \) as \( \varepsilon \to 0 \) for \( x \in [0, \infty) \).  
		Applying the Sobolev embedding theorem and \( L^p \) interior estimates to \eqref{eq:part1}, for large \( p > 1 \) and \( \alpha \in (0,1) \), we obtain  
		\[
		\| q_c^1(\cdot;\varepsilon) \|_{C^{1+\alpha}([M, M+1])} \leq C_1 \| q_c^1(\cdot;\varepsilon) \|_{W^{2,p}([M,M+1])} \leq C_2,
		\]  
		with constants \( C_1, C_2 \) independent of \( \varepsilon \) and \( M \). Combining with condition $({\bf f_m})$ and Schauder estimates, we get \( \| q_c^1(\cdot;\varepsilon) \|_{C^{2+\alpha}([0, M])} \leq C(M) \), where \( C(M) \) depends on \( M \) but not \( \varepsilon \). Hence \( q_c^1(\cdot;\varepsilon) \to Q_1 \) in \( C^{2}_{\mathrm{loc}}([0, \infty)) \) as \( \varepsilon \to 0 \) and \( Q_1 \) satisfies  
		\[
		d Q_1'' - c Q_1' + f(Q_1) = 0 \quad \text{for} \quad x > 0
		\]  
		with \( Q_1'(x) \leq 0 \) for \( x \in [0, \infty) \). Then, \( Q_1^* := \lim_{x\to \infty} Q_1(x) \) exists and \( Q_1^* \geq 1 \).  
		As in the proof of Theorem \ref{p1}, we have \( \lim_{x \to \infty} Q_1'(x) = 0 \), implying \( \lim_{x \to \infty} Q_1''(x) = -\frac{1}{d} f(Q_1^*) \). Since \( \lim_{x \to \infty} Q_1''(x) = 0 \), it follows that \( f(Q_1^*) = 0 \), which implies \( Q_1^* = 1 \). Thus,  
		\begin{equation*}
			\begin{cases}
				d Q_1'' - c Q_1' + f(Q_1) = 0, & x > 0, \\
				Q_1(x) \geq 1, \quad Q_1'(x) \leq 0, & x \in [0, \infty), \\
				Q_1(0) = \delta, \  Q_1(\infty) = 1.
			\end{cases}
		\end{equation*}  
		By uniqueness of the solution to \eqref{eq:part1}, we deduce \( Q_1 = q_{c} \). Thus,  
		\[
		\lim_{\varepsilon \to 0} q_c^1(x;\varepsilon) = q_c(x) \quad \text{in} \quad C^2_{\mathrm{loc}}([0, \infty)).
		\]  
		The same conclusion holds for \( q_c^2(x;\varepsilon) \).

			The above convergence, together with the monotonicity of \( q_c^i(x;\varepsilon) \) in \( \varepsilon \) and the limits \( q_c^1(\infty;\varepsilon) = 1 - \varepsilon \), \( q_c^2(\infty;\varepsilon) = 1 + \varepsilon \), yields  
			\[
			\lim_{\varepsilon \to 0} \| q_c^i(\cdot;\varepsilon) - q_c \|_{L^\infty([0, \infty))} = 0, \quad i = 1, 2.
			\]
			This completes the first part of the proof.

		(ii) Theorem \ref{p1} implies that for each \( i = 1, 2 \), the problem \eqref{eq:part2} with \( f \) replaced by \( f_{\varepsilon}^i \) admits a unique solution pair \( (c,q)=(c_i^*(\varepsilon), q_i^*(x;\varepsilon)) \), and \(\lim_{x \to \infty} q_1^*(x;\varepsilon) = 1 - \varepsilon\), \(\lim_{x \to \infty} q_2^*(x;\varepsilon) = 1 + \varepsilon\).

		Let us denote the solution of \eqref{s3a} with \( f \) replaced by \( f_{\varepsilon}^i \)  by $P_{c}^{i}(\delta;\varepsilon)$ for $i=1,2$.
		For any \( c \in \mathbb{R} \), Lemma~\ref{l1} implies
		\[
		P_{c}^{1}(\delta;\varepsilon) < P_{c}(\delta) < P_{c}^{2}(\delta;\varepsilon).
		\]
		Direct calculation yields
		\[
		\xi_{\varepsilon}^1(c) := P_{c}^{1}(\delta;\varepsilon) - \frac{\delta}{d}c 
		\;<\; 
		\xi(c) := P_{c}(\delta) - \frac{\delta}{d}c 
		\;<\; 
		\xi_{\varepsilon}^2(c) := P_{c}^{2}(\delta;\varepsilon) - \frac{\delta}{d}c.
		\]
		Since \(\xi(c^{*}) = 0\), we have \(\xi_{\varepsilon}^1(c^{*}) < 0 < \xi_{\varepsilon}^2(c^{*})\). Given that \(\xi_{\varepsilon}^i(c)\) is strictly decreasing in \(c\) for \(i = 1, 2\) (as established in the proof of Proposition~\ref{p1}), we conclude that 
		\[
		c_1^{*}(\varepsilon) < c^* < c_2^{*}(\varepsilon).
		\]
		Moreover, Lemma~\ref{l1} implies \(P_{c}^{i}(\delta;\varepsilon)\) is nonincreasing in \(\varepsilon\) for \(i = 1, 2\), a similar monotonicity argument shows that \(c_{2}^{*}(\varepsilon)\) is increasing and \(c_{1}^{*}(\varepsilon)\) is decreasing in \(\varepsilon\). The continuity of \(\xi_{\varepsilon}^i(c)\) in \(\varepsilon\) (for \(i = 1, 2\)) and the uniqueness of \(c^{*}\) as the solution to \(\xi(c) = 0\) now imply 
		\[
		\lim_{\varepsilon \to 0} c_{i}^{*}(\varepsilon) = c^*, \quad i = 1, 2.
		\]

		Since \( c_{1}^{*}(\varepsilon) \) decreases with \(\varepsilon\), for \(\varepsilon_1 > \varepsilon_2\), Theorem~\ref{p1} and Corollary~\ref{l2}~(i) imply
		\[
		q_1^*(x;\varepsilon_1) = q_{c^*(\varepsilon_1)}^1(x;\varepsilon_1) < q_{c^*(\varepsilon_2)}^1(x;\varepsilon_1) < q_{c^*(\varepsilon_2)}^1(x;\varepsilon_2) = q_1^*(x;\varepsilon_2) \quad\mbox{for } x > 0.
		\]
		Hence \( q_1^*(x;\varepsilon) \) is decreasing in \( \varepsilon > 0 \). Since \( 1 - \varepsilon \leq q_1^*(x;\varepsilon) \leq \delta \), there exists \( Q^1(x) \) with \( 1 \leq Q^1(x) \leq \delta \) such that \( q_1^*(x;\varepsilon) \to Q^1(x) \) as \(\varepsilon \to 0\) for \( x \in [0,\infty) \).  
	By an argument analogous to that in~(i), we have \( q_1^*(\cdot;\varepsilon) \to Q^1 \) in \( C^{2}_{\mathrm{loc}}([0,\infty)) \) as \( \varepsilon \to 0 \), where \( Q^1 \) satisfies
		\begin{equation*}
			\begin{cases}
				d Q_1'' - c^* Q_1' + f(Q_1) = 0, & x > 0, \\
				Q_1(x) \geq 1, \quad Q_1'(x) \leq 0, & x \geq 0, \\
				Q_1(0) = \delta, \quad  Q_1(\infty) = 1.
			\end{cases}
		\end{equation*}
		By uniqueness, \( Q^1 = q_{c^*} \). Therefore,
		\[
		\lim_{\varepsilon \to 0} q_i^*(\cdot;\varepsilon) = q^* \quad \text{in } C^{2}_{\mathrm{loc}}([0,\infty)), \quad i=1,2.
		\]
		This convergence, combined with the monotonicity of \( q_i^*(\cdot;\varepsilon) \) in \(\varepsilon\) and the limits \( q_1^*(\infty) = 1 - \varepsilon \), \( q_2^*(\infty) = 1 + \varepsilon \), implies
		\[
		\lim_{\varepsilon \to 0} \| q_i^*(\cdot;\varepsilon) - q^* \|_{L^\infty([0,\infty))} = 0, \quad i=1,2.
		\]
		This completes the proof. \qed

\subsection{Proof of Lemma \ref{l4}}
		We present the proof for the case \( \bar c_0 > c^* \), as the case \( \underline c_0 < c^* \) follows by a similar argument.
		
		Since \( \bar c_1 < \bar c_0 \), Lemma~\ref{l1} ensures that $\bar q_1'(0) = P_{\bar c_1}(\delta) < P_{\bar c_0}(\delta) = \bar q_0'(0)$.
		It then follows that
		\begin{align*}
			\bar c_2 = \frac{d}{\delta} \bar q_1'(0) + \frac{1}{M+1} < \frac{d}{\delta} \bar q_0'(0) + \frac{1}{M} = \bar c_1\quad {\rm and}\quad		\bar c_1=\frac{d}{\delta} \bar q_0'(0) + \frac{1}{M} > \frac{d}{\delta} \bar q_0'(0) > \frac{d}{\delta} \bar q_1'(0).
		\end{align*}
		By induction, the sequence \( \{\bar c_n\} \) is strictly decreasing and satisfies
		\[
		\bar c_n > \frac{d}{\delta} \bar q_n'(0) \ = \frac{d}{\delta} q_{\bar c_n}'(0) \quad \text{for all } n \geq 0.
		\]
		Since \( \xi(c^*) = 0 \), and \( \xi(c) = q_c'(0) - \frac{\delta}{d} c \) is strictly decreasing in \( c \in \mathbb{R} \), we conclude that \( \bar c_n > c^* \) for all \( n \geq 0 \).
		
		Because \( \{\bar c_n\} \) is decreasing and bounded from below by $c^*$, there exists a constant \( C \geq c^* \) such that \( \bar c_n \to C \) as \( n \to \infty \). By Theorem~\ref{p1},  \( \bar q_n \to Q \) in \( C^{2}_{\mathrm{loc}}([0, \infty)) \), with \( Q \) satisfying
		\begin{equation}\label{eq:limit}
			\begin{cases}
				dQ'' - C Q' + f(Q) = 0, & x \in (0, \infty), \\
				Q(0) = \delta, \quad Q(\infty) = 1. \\
			\end{cases}
		\end{equation}
		Furthermore,
		\[
		\frac{d}{\delta} Q'(0) = \lim_{n \to \infty} \frac{d}{\delta} \bar q_n'(0) = \lim_{n \to \infty} \left( \bar c_{n+1} - \frac{1}{M+n} \right) = C.
		\]
		By the uniqueness of \( c^* \) established in Theorem~\ref{p1}~(ii), we deduce   \( C = c^* \) and \( Q = q^*=q_{c^*} \).
		
		The second identity in \eqref{2.14a} then follows from the monotonicity of \( \bar q_n \) in \( n \) (as shown in Theorem~\ref{p1}~(i)) and the convergence \( \bar q_n \to Q \) in \( C^2_{\mathrm{loc}}([0, \infty)) \), along with the boundary condition
		\[
		\bar q_n(\infty) = Q(\infty) = 1.
		\]
		This completes the proof. \qed

\end{document}